\documentclass[a4paper,oneside,onecolumn]{article}

\usepackage[utf8]{inputenc} % Required for including letters with accents
\usepackage[T1]{fontenc}    % Use 8-bit encoding that has 256 glyphs
\usepackage{textcomp}       % Better font shape coverage
\usepackage[left=2.0cm,top=2.5cm,right=2.0cm,bottom=2.5cm,bindingoffset=0.0cm]{geometry}

\usepackage{amsmath,amssymb,amsthm,amsfonts}
\usepackage{mathtools} % For \coloneqq and other math symbols
\usepackage{stmaryrd}  % For special math symbols (stochastic intervals)

\DeclareMathAlphabet{\mathpzc}{OT1}{pzc}{m}{it}
\SetSymbolFont{stmry}{bold}{U}{stmry}{m}{n}

\usepackage{bbm}       % For serif-style double-struck "1" \mathbbm{1}
\usepackage{mathrsfs}
\usepackage{dsfont}
\usepackage{cancel}
\allowdisplaybreaks

\usepackage{graphics,graphicx,epstopdf}  % Required for including images
\usepackage{float}     % For the table positioning
\usepackage{multirow}
\usepackage{booktabs}  % For professional-looking tables
\usepackage{caption}   % For table and figure captions
\usepackage{subfig}    % Required for creating figures with multiple parts (subfigures)
\usepackage{color, xcolor}
\usepackage{soul}

\usepackage{enumitem}  % Required for manipulating the whitespace between and within lists % required for [label=(\roman*)]

\usepackage{setspace}  % Ensure this package is included
\usepackage{lineno}    % Adds consecutive line numbers to the document margins
\usepackage{authblk}   % For including multiple authors

\theoremstyle{plain} % Define theorem styles here based on the plain style (used for theorems, lemmas, propositions
\newtheorem{theorem}{Theorem}[section]
\newtheorem{lemma}{Lemma}[section]

\newtheorem{assumption}{Assumption}[section]

\newtheorem{example}{Example}[section]

\theoremstyle{remark} % Define theorem styles here based on the remark style (used for remarks and notes)
\newtheorem{remark}{Remark}[section]

\definecolor{c1}{rgb}{0,0,1}      % blue
\definecolor{c2}{rgb}{0,0.5,1}    % light blue
\definecolor{c3}{rgb}{0.7,0,0.8}  % red blue

\usepackage{hyperref} % Enables standard clickable hyperlinks in the PDF
\hypersetup{
	breaklinks = true,       % will make the text rather than the page numbers of the ToC the active link
	linktocpage=true,        % will make the page numbers rather than the text of the ToC the active link
	colorlinks=true,         % false: boxed links; true: colored links
	linkcolor={c1},          % color of internal links
	citecolor={c2},          % color of links to bibliography
	filecolor=black,         % color of file links
	urlcolor={c3},           % external links/urls
}
\numberwithin{equation}{section}

\makeatletter
\def\namedlabel#1#2{\begingroup
    #2%
    \def\@currentlabel{#2}%
    \phantomsection\label{#1}\endgroup
}
\makeatother

\begin{document}
%----------------------------------------------------------------------------------------
%	TITLE SECTION
%----------------------------------------------------------------------------------------

% \title{\textbf{Optimal Error Estimates of a Finite Element Method for Semilinear SPDEs with Additive Noise and Nonsmooth Initial Data}}

% \title{\textbf{Optimal Error Estimates of a Finite Element Method for Semilinear Stochastic Convection-Diffusion Equations with Additive Noise and Nonsmooth Initial Data}}

% \title{\textbf{Sharp Error Estimates for Fully Discrete Approximations of Semilinear Stochastic Convection-Diffusion Equations with Additive Noise and Nonsmooth Initial Data}}

\title{\textbf{Sharp Error Estimates for a Fully Discrete Approximation of Semilinear Stochastic Convection--Diffusion Equations with Additive Noise and Nonsmooth Initial Data}}

\author[1]{\textsc{Jitendra Nath Naik}\thanks{\href{mailto:jitendra20232201@iitgoa.ac.in}{jitendra20232201@iitgoa.ac.in}}}
\author[1]{\textsc{Lok Pati Tripathi}\thanks{Corresponding author: Lok Pati Tripathi (\href{mailto:lokpati@iitgoa.ac.in}{lokpati@iitgoa.ac.in})}}
\affil[1]{School of Mathematics and Computer Science, Indian Institute of Technology Goa, Goa 403401, India.}

\date{} % Leave empty to omit a date

\maketitle % Print the title

%----------------------------------------------------------------------------------------
%	ARTICLE SECTION
%----------------------------------------------------------------------------------------

% \begin{abstract}
% We establish sharp strong error estimates for the fully discrete approximation of semilinear stochastic convection-diffusion equations driven by additive noise subject to nonsmooth initial data. The problem is discretized in space using a standard Galerkin finite element method and in time via the linearly implicit Euler scheme. By decoupling the noise regularity index $\beta \in (0, 2]$ from the initial data regularity index $\mu \in (0, 2]$, we derive time-weighted spatial and temporal regularity bounds that isolate the initial singularity at $t = 0$ via an integrable temporal weight. In the general convection-diffusion setting, where the drift operator $F$ maps into negative-order fractional spaces, we prove sharp strong convergence rates of order $O(h^\beta + k^{\beta/2})$ for $\beta \in (0, 2)$ and $O\bigl(h^2 + k(1+\ln(t_n/k))^{1/2}\bigr)$ at the critical endpoint $\beta = 2$ for any fixed $t_n > 0$, without requiring the initial data to match the spatial regularity of the noise. Furthermore, these convergence rates are achieved under a relaxed differentiability condition requiring only a Lipschitz continuous first G\^ateaux derivative of $F$. Numerical experiments validate the theoretical results across both trace-class and space-time white noise regimes.
% \end{abstract}

\begin{abstract}
We establish sharp strong error estimates for a fully discrete approximation of semilinear stochastic convection--diffusion equations driven by additive noise subject to nonsmooth initial data. The problem is discretized in space using a standard Galerkin finite element method and in time via the linearly implicit Euler scheme. By decoupling the noise regularity index $\beta \in (0, 2]$ from the initial data regularity index $\mu \in (0, 2]$, we derive time-weighted spatial and temporal regularity bounds that isolate the initial singularity at $t = 0$ into an integrable temporal weight. By allowing the drift operator $F$ to map into negative-order fractional spaces to incorporate the convection term, we prove sharp strong convergence rates of order $O(h^\beta + k^{\beta/2})$ for $\beta \in (0, 2)$ and $O\bigl(h^2 + k(1+\ln(t_n/k))^{1/2}\bigr)$ at the critical endpoint $\beta = 2$ for any fixed $t_n > 0$, without requiring the regularity of the initial data to match that of the noise. Furthermore, these convergence rates are achieved under a relaxed differentiability condition requiring only a Lipschitz continuous first G\^ateaux derivative of $F$. Numerical experiments validate the theoretical results across both trace-class and space-time white noise regimes.
\end{abstract}


\vspace{1em}

% 60H35, 65M15, 65M60

\vspace{2em} % Adds some space before the table of contents
% \tableofcontents % <--- ADD THIS COMMAND HERE
% \vspace{2em} % Adds some space after the table of contents

%%%%%%%%%%%%%%%%%%%%%%%%%%%%%%%%%%%%%%%%%%%%%%
\section{Introduction}\label{sec:introduction}
%%%%%%%%%%%%%%%%%%%%%%%%%%%%%%%%%%%%%%%%%%%%%%
Let $\mathcal{O} \subset \mathbb{R}^d$ ($d \in \{1,2,3\}$) be a bounded convex polytopal domain with boundary $\partial\mathcal{O}$. Given a probability space $(\Omega, \mathcal{F}, \mathbb{P})$ equipped with a normal filtration $\{\mathcal{F}_t\}_{t \in [0,T]}$, we consider the semilinear stochastic convection--diffusion equation driven by additive noise:
\begin{equation}\label{eq:SPDE}
\left\{
\begin{aligned}
    du(t,x) &= \big[ \nabla \cdot (\boldsymbol{a}(x)\nabla u(t,x)) - \boldsymbol{b}(x) \cdot \nabla u(t,x) - c(x)u(t,x) + f(t, x, u(t,x)) \big] dt \\
            &\quad + g(t, x) \, dW(t,x), && (t,x) \in (0,T] \times \mathcal{O}, \\
    u(t,x) &= 0, && (t,x) \in (0,T] \times \partial\mathcal{O}, \\
    u(0,x) &= u_0(x), && x \in \mathcal{O}.
\end{aligned}
\right.
\end{equation}
Here, $W$ is a (possibly cylindrical) $Q$-Wiener process on $U = L^2(\mathcal{O})$ associated with a linear, bounded, self-adjoint, positive semidefinite covariance operator $Q \colon U \to U$. The coefficients $\boldsymbol{a} \colon \mathcal{O} \to \mathbb{R}^{d \times d}$, $\boldsymbol{b} \colon \mathcal{O} \to \mathbb{R}^{d}$, and $c \colon \mathcal{O} \to \mathbb{R}$ are sufficiently smooth and bounded, with $\boldsymbol{a}(x)$ symmetric and uniformly elliptic, i.e., $y^T \boldsymbol{a}(x) y \ge a_0 |y|^2$ for some $a_0 > 0$ and all $x \in \mathcal{O}, y \in \mathbb{R}^d$, while $f \colon [0,T] \times \mathcal{O} \times \mathbb{R} \to \mathbb{R}$ and $g \colon [0,T] \times \mathcal{O} \to \mathbb{R}$ are measurable.

To cast \eqref{eq:SPDE} into an abstract Hilbert space framework, let $H = L^2(\mathcal{O})$ and denote by $U_0 \coloneqq Q^{1/2}(U)$ the reproducing kernel Hilbert space endowed with the inner product $(\phi_0, \psi_0)_{U_0} \coloneqq (Q^{-1/2}\phi_0, Q^{-1/2}\psi_0)_U$, where $Q^{-1/2}$ denotes the pseudoinverse of $Q^{1/2}$ in the case where $Q$ is not injective. Let $\mathrm{HS}(U_0, H)$ be the space of Hilbert--Schmidt operators from $U_0$ to $H$ with norm $\|\Phi\|_{\mathrm{HS}(U_0, H)} \coloneqq \bigl(\sum_{m=1}^\infty \|\Phi \psi_m\|_{H}^2\bigr)^{1/2}$ for an orthonormal basis $\{\psi_m\}_{m \ge 1}$ of $U_0$ (cf.~\cite{MR3236753, MR2329435}). We introduce the second-order elliptic differential operator $A \colon \mathcal{D}(A) \subset H \to H$ defined via $A v = -\nabla \cdot (\boldsymbol{a}(x) \nabla v)$ subject to homogeneous Dirichlet boundary conditions. It is well known that $-A$ is the infinitesimal generator of an analytic $\mathcal{C}_0$-semigroup of contractions on $H$, which we denote by $S(t) = e^{-tA}$ for $t \ge 0$. By applying the spectral theorem to $A^{-1}$, we obtain an orthonormal basis for $H$ consisting of eigenfunctions of $A$. This enables us to define the fractional powers $A^\gamma$ for any $\gamma \in \mathbb{R}$ alongside their corresponding fractional domain spaces $\dot{H}^\gamma \coloneqq \mathcal{D}(A^{\gamma/2})$, equipped with the norm $\|u\|_{\dot{H}^\gamma} \coloneqq \|A^{\gamma/2}u\|_H$ (cf. Kruse~\cite[Appendix~B.2]{MR3154916}). Given a noise regularity parameter $\beta \in (0,2]$, we define the drift and diffusion operators $F \colon [0,T] \times H \to \dot{H}^{-1}$ and $G \colon [0,T] \to \mathrm{HS}(U_0, \dot{H}^{\beta-1})$ by 
\begin{equation}\label{eq:Nemytskii_FG}
    F(t, \phi) \coloneqq f(t, \cdot, \phi) - \boldsymbol{b} \cdot \nabla \phi - c \phi, \qquad
    G(t)\psi \coloneqq g(t, \cdot) \psi.
\end{equation}
Problem \eqref{eq:SPDE} then takes the abstract stochastic evolution equation form for $X(t) \coloneqq u(t, \cdot)$:
\begin{equation}\label{eq:SEE}
    dX(t) + A X(t) \, dt = F(t, X(t)) \, dt + G(t) \, dW(t), \quad t \in (0, T], \qquad X(0) = X_0.
\end{equation}
The existence, uniqueness, and regularity of the mild solution to \eqref{eq:SEE} are well established in the literature (see, e.g., \cite{MR2329435, MR2856611, MR3154916, MR3236753, MR3308418, MR3410409, MR2852200, MR2968672}). Under the structural hypotheses stated in Section~\ref{sec:assumptions}, equation \eqref{eq:SEE} admits a unique $H$-valued mild solution~\cite{NaikTripathi2026} satisfying
\begin{equation}\label{eq:mild_solution_SEE}
    X(t) = S(t)X_{0} + \int_{0}^{t} S(t-\sigma)F(\sigma,X(\sigma))\,d\sigma + \int_{0}^{t} S(t-\sigma)G(\sigma)\,dW(\sigma), \quad t \in [0,T].
\end{equation}

Numerical approximations of semilinear parabolic stochastic partial differential equations (SPDEs) via finite difference, finite element, and spectral Galerkin methods have been widely studied over the past two decades~\cite{MR1637047, MR1683281, MR1873517, MR1953619, MR2182132, MR4050540}. In the specific context of additive noise, the literature on spatial semidiscretizations analyzes strong errors for both smooth and nonsmooth initial data~\cite{MR2211047, MR3168284, MR3872387}. However, the theoretical convergence order in these standard works is dictated by the spatial regularity of the initial condition; thus, the sharp $O(h^\beta)$ spatial rate of convergence is achieved only when the initial state is sufficiently smooth ($X_0 \in L^p(\Omega; \dot{H}^\beta)$). Despite the extensive developments in fully discrete schemes (see, e.g., \cite{MR2182132, MR2600932, MR3168284, MR3327065, MR3872387}), a central challenge in strong convergence analysis is overcoming the temporal convergence barrier of $O(k^{1/2})$, dictated by the $1/2$-H\"older continuity of the solution; for instance, Kruse~\cite{MR3168284} established $O(h^\beta + k^{1/2})$ for $\beta \in [1,2]$ by assuming a global Lipschitz condition on the drift into a negative-order fractional space. To achieve higher-order convergence, Kruse~\cite{MR3274889} studied a Milstein--Galerkin finite element scheme, which reduces to the linearly implicit Euler method for additive noise, and obtained a temporal convergence rate of $O(k^{\frac{1+r}{2}})$ for a parameter $r \in [0, 1)$ characterizing the spatial regularity of the initial state and noise. Of particular relevance to the present study, Wang~\cite{MR3649432} adapted the Taylor expansion technique of Kloeden et al.~\cite{MR2728063} to the standard FEM and linearly implicit Euler scheme, establishing the sharp temporal rate $O(k^{\beta/2})$ under the assumption that the drift operator $F$ is twice differentiable. However, achieving the optimal temporal order ($\beta=2$) in \cite{MR3649432} demands that the initial data satisfy $X_0 \in L^p(\Omega;\dot{H}^2)$, which is overly restrictive for physical applications. Even recent studies on fully discrete schemes for non-globally Lipschitz models, such as the stochastic Allen--Cahn equation, require the initial data regularity to match the noise regularity to avoid a theoretical order reduction (see, e.g.,~\cite{MR3977202, MR4140034, MR4628763, MR4858504, 10.1093/imanum/draf152}). In physical applications involving discontinuous phase interfaces or localized thermal sources, the initial datum $u_0$ is often nonsmooth (e.g., an indicator function), which reduces the theoretical spatial and temporal order of convergence. For deterministic parabolic PDEs, finite element analyses (see, e.g.,~\cite{MR906175, MR906176, MR970700, MR2249024}) have shown that exploiting the smoothing property of the analytic semigroup overcomes this theoretical order reduction for $t > 0$ via time-weighted error estimates. However, recovering optimal convergence rates for SPDEs with nonsmooth initial data is nontrivial. In our companion study of SPDEs with multiplicative noise~\cite{naik2026sharperrorestimatesfully}, we utilize a parameter-decoupling strategy for nonsmooth initial data, obtaining $O(h^{\beta-\varepsilon} + k^{\frac{1}{2}\min\{\beta-\varepsilon,1\}})$ for $\mu > \beta - 1$.

In this article, we develop a sharp strong convergence theory for the finite element and linearly implicit Euler approximation of semilinear stochastic convection--diffusion equations driven by additive noise subject to nonsmooth initial data $X_0 \in L^{2p}(\Omega;\dot{H}^\mu)$ with $\mu \in (0,2]$. First, by decoupling the noise regularity parameter $\beta \in (0,2]$ from the initial data regularity parameter $\mu \in (0,2]$ and removing the standard structural constraint $\beta = \mu$ prevalent in the existing literature~\cite{MR3168284, MR3274889, MR3649432, MR3872387}, we establish sharp time-weighted spatial and temporal regularity bounds for the mild solution (Theorem~\ref{thm:singular_regularity}), isolating the exact loss of initial regularity into an integrable temporal weight $t^{-(\beta-\nu)/2}$ ($\nu \coloneqq \min\{\beta,\mu\}$). Second, we relax both the temporal regularity and differentiability hypotheses on the drift $F$: we consider time-dependent drifts and require only temporal H\"older continuity of order $\beta/2$ for $\beta \in (0,2)$ rather than the standard global Lipschitz continuity in time; moreover, rather than imposing twice differentiability~\cite{MR3327065, MR3649432, MR3872387} or differentiability with $F'$ Lipschitz continuous into $\mathcal{L}(H, \dot{H}^{-1+r})$ ($r \in [0,1)$)~\cite{MR3274889}, we require a Lipschitz continuous first G\^ateaux derivative $F'$ as a mapping into $\mathcal{L}(H, \dot{H}^{-\vartheta})$ for some $\vartheta \in [0,2)$. By leveraging the Sobolev embedding $L^1(\mathcal{O}) \hookrightarrow \dot{H}^{-\vartheta}$ for $\vartheta > d/2$ ($d \in \{1,2,3\}$), this relaxation enables us to include physically relevant $\mathcal{C}^{1,1}$ scalar nonlinearities (i.e., $\mathcal{C}^1$ functions with Lipschitz continuous first derivatives) whose second derivatives fail to exist, such as the modified Langmuir fractional surface coverage function $f(u) = \frac{u}{1+|u|}$ (which describes the adsorption of molecules from a fluid to a solid surface \cite{langmuir1918adsorption, 10.1093/imanum/draf152}), the nonlinearity $f(u) = \frac{u|u|}{1 + u^2}$, and regularized truncation functions $f(u) = \bigl(u - \frac{u|u|}{2\varepsilon}\bigr)\mathbbm{1}_{|u|\le\varepsilon} + \frac{\varepsilon}{2}\operatorname{sign}(u)\mathbbm{1}_{|u|>\varepsilon}$ for $\varepsilon > 0$. Third, under these relaxed conditions in the general convection--diffusion setting (where $F$ and $F'$ map into $\dot{H}^{-1}$), we establish strong error bounds of order $O\bigl(h^\beta + k^{\beta/2}\ell_{n}^{(\beta)}\bigr)$ for any fixed $t_n > 0$ and $\beta \in (0,2]$ (Theorem~\ref{thm:strong_convergence_of_fully_discrete_scheme}), achieving the sharp convergence rate $O(h^\beta + k^{\beta/2})$ for $\beta \in (0,2)$ and $O\bigl(h^2 + k(1+\ln(t_n/k))^{1/2}\bigr)$ at the critical endpoint $\beta = 2$ away from $t=0$. Finally, when the convection term vanishes ($\boldsymbol{b} \equiv \boldsymbol{0}$) and $F$ (along with $F'$) maps into $H$, we eliminate the logarithmic factor at $\beta = 2$ under an analogous first-order G\^ateaux differentiability condition that remains weaker than the assumptions in~\cite{MR3649432}, recovering the optimal convergence rate $O(h^2 + k)$ for $t_n > 0$ (Theorem~\ref{thm:optimal_strong_convergence}).

The remainder of the article is organized as follows. Section~\ref{sec:assumptions} formulates our structural assumptions and establishes well-posedness. Section~\ref{sec:fullydiscrete_approximation} formulates the fully discrete scheme and states our main convergence theorem (Theorem~\ref{thm:strong_convergence_of_fully_discrete_scheme}). Section~\ref{sec:regularity_results} establishes the time-weighted regularity estimates for the mild solution, which are used in Section~\ref{sec:proof_of_fully_discrete_result} to prove Theorem~\ref{thm:strong_convergence_of_fully_discrete_scheme}. Section~\ref{sec:optimal_error_estimates} establishes the logarithm-free optimal error estimate at the critical endpoint $\beta = 2$ in the absence of the convection term, and Section~\ref{sec:numerical_results} provides numerical experiments confirming our theoretical predictions.

%%%%%%%%%%%%%%%%%%%%%%%%%%%%%%%%%%%%%%%%%%%%%%%%%%%%%%%%%%%%%%%%%%%%%%%%%%%%
\section{Assumptions and well-posedness}\label{sec:assumptions}
%%%%%%%%%%%%%%%%%%%%%%%%%%%%%%%%%%%%%%%%%%%%%%%%%%%%%%%%%%%%%%%%%%%%%%%%%%%%

We retain the framework established in Section~\ref{sec:introduction} and write $\mathcal{L}(H)$ for the Banach space of bounded linear operators on $H$ endowed with the usual operator norm $\|\cdot\|_{\mathcal{L}(H)}$. Throughout the paper, $C > 0$ denotes a generic constant that may vary from one occurrence to the next; any crucial dependence on parameters will be stated explicitly.

As a consequence of the settings in Section~\ref{sec:introduction}, the second-order elliptic operator $A \colon \mathcal{D}(A) \subset H \to H$ is densely defined, self-adjoint, positive definite, and possesses a compact inverse. To ensure the well-posedness of \eqref{eq:SEE} and to carry out the finite element error analysis, we impose the following hypotheses.

\begin{assumption}\label{ass:initial_data}
Let $p \ge 2$ and $\mu \in (0,2]$. Let the initial data $X_0 \colon \Omega \to H$ be $\mathcal{F}_0/\mathcal{B}(H)$-measurable and satisfy
\[
    X_0 \in L^{2p}\bigl(\Omega; \dot{H}^{\mu}\bigr).
\]
\end{assumption}

\begin{assumption}\label{ass:diffusion}
Let $\beta \in (0,2]$ characterize the spatial regularity of the noise. There exists a constant $C > 0$ such that the mapping $G \colon [0,T] \to \mathrm{HS}(U_0, \dot{H}^{\beta-1})$ satisfies
\begin{align}
\|G(t)\|_{\mathrm{HS}(U_0, \dot{H}^{\beta-1})} &\leq C, \\
\|G(t)-G(s)\|_{\mathrm{HS}(U_0, \dot{H}^{\beta-1})} &\leq C \, |t-s|^\delta,
\end{align}
for all $t, s \in [0,T]$ and some $\delta \in [\beta/2, 1]$.
\end{assumption}

\begin{assumption}\label{ass:drift}
Let $\beta \in (0,2]$ be as in Assumption~\ref{ass:diffusion}. There exists a constant $C > 0$ such that the drift mapping $F \colon [0,T] \times H \to \dot{H}^{-1}$ satisfies the following conditions:
\begin{enumerate}[label={\upshape(\roman*)}]
    \item \emph{(Lipschitz continuity and linear growth in $\dot{H}^{-1}$).} For all $t, s \in [0,T]$ and $u, v \in H$,
    \begin{align}
        \|F(t,u) - F(s,v)\|_{\dot{H}^{-1}} &\le C \bigl( |t-s|^{\beta/2} + \|u-v\|_H \bigr), \label{eq:F_Lipschitz_H1} \\
        \|F(t,v)\|_{\dot{H}^{-1}} &\le C\bigl(1 + \|v\|_H\bigr). \label{eq:F_linear_growth}
    \end{align}

    \item \emph{(Lipschitz continuity and linear growth in $\dot{H}^{\beta-2}$ for $\beta \in (1,2]$).} When $\beta \in (1,2]$, the restriction of $F$ to $[0,T] \times \dot{H}^{\beta-1}$ maps into $\dot{H}^{\beta-2}$ and satisfies, for all $t, s \in [0,T]$ and $u, v \in \dot{H}^{\beta-1}$,
    \begin{align}
        \|F(t,u) - F(s,v)\|_{\dot{H}^{\beta-2}} &\le C \bigl( |t-s|^{\beta/2} + \|u-v\|_{\dot{H}^{\beta-1}} \bigr), \label{eq:F_Lipschitz_H_beta} \\
        \|F(t, v)\|_{\dot{H}^{\beta-2}} &\le C\bigl(1 + \|v\|_{\dot{H}^{\beta-1}}\bigr). \label{eq:F_growth_H_beta}
    \end{align}

    \item \emph{(G\^ateaux differentiability).} The mapping $F(t, \cdot) \colon H \to \dot{H}^{-1}$ is G\^ateaux differentiable, and there exists $\vartheta \in [0,2)$ such that its derivative $F'(t,u) \coloneqq \partial_u F(t,u)$ satisfies, for all $t \in [0,T]$ and $u, v, w \in H$,
    \begin{align}
        \|F'(t,u)v\|_{\dot{H}^{-1}} &\le C \|v\|_H, \label{eq:F_prime_bound_H1} \\
        \|\bigl(F'(t,u) - F'(t,v)\bigr)w\|_{\dot{H}^{-\vartheta}} &\leq C \,\|u-v\|_H \|w\|_H. \label{eq:F_prime_lipschitz_H1}
    \end{align}
\end{enumerate}
\end{assumption}

\begin{remark}\label{rem:drift_regularity}
In Assumption~\ref{ass:drift}, we impose a temporal H\"older continuity condition of order $\beta/2$ for $\beta \in (0,2)$ on the drift $F$, rather than the global Lipschitz condition in time traditionally assumed in standard monographs (see, e.g., Pazy~\cite[Chapter~6]{MR710486}) and recent literature (see, e.g., \cite{MR4126762, MR4927014}). Since the temporal convergence rate of the scheme is inherently constrained to $O(k^{\beta/2})$ by the stochastic convolution, requiring higher temporal regularity for the drift coefficient is stronger than necessary. By aligning the regularity of the drift coefficient with that of the stochastic convolution, we expand the class of admissible time-dependent drift coefficients while preserving the optimal strong convergence rate of $O(h^\beta + k^{\beta/2})$.
\end{remark}

\begin{theorem}[Well-posedness and base regularity]\label{thm:base_regularity}
Under Assumptions~\ref{ass:initial_data}, \ref{ass:diffusion}, and~\ref{ass:drift}{\upshape(i)}, problem~\eqref{eq:SEE} admits a unique mild solution $X$ given by \eqref{eq:mild_solution_SEE}. Moreover, for any $q \in [2, 2p]$, there exist constants $C_1, C_2 > 0$ (depending on $T$, with $C_2$ also depending on $X_0$) such that
\begin{enumerate}[label={\upshape(\roman*)}]
    \item $\displaystyle \sup_{t \in [0,T]} \|X(t)\|_{L^q(\Omega; H)} \le C_1 \bigl( 1 + \|X_0\|_{L^q(\Omega; H)} \bigr)$, \label{eq:base_spatial_regularity}
    \item $\displaystyle \|X(t) - X(s)\|_{L^q(\Omega; H)} \le C_2 \, (t-s)^{\tilde{\delta}}$ for all $0 \le s < t \le T$ and $0 < \tilde{\delta} < \frac{1}{2}\min\{1, \mu, \beta\}$. \label{eq:base_temporal_regularity}
\end{enumerate}
\end{theorem}

\begin{proof}
Since $\dot{H}^\mu \hookrightarrow H$ continuously, Assumption~\ref{ass:initial_data} ensures $X_0 \in L^q(\Omega; \dot{H}^\mu) \hookrightarrow L^q(\Omega; H)$ for all $q \in [2, 2p]$. Setting the abstract state space in \cite[Theorem~2.1]{NaikTripathi2026} to $H$ (with initial regularity increment $\epsilon = \min\{1, \mu\}$, drift codomain index $1$, diffusion codomain index $(1-\beta)^+ \in [0,1)$, and no temporal weights on the coefficients), the properties of $A$ along with Assumptions~\ref{ass:initial_data}, \ref{ass:diffusion}, and~\ref{ass:drift}{\upshape(i)} fulfill all hypotheses of \cite[Theorem~2.1]{NaikTripathi2026}, from which parts~(i) and~(ii) follow immediately via \cite[Theorem~2.1(i) and~(iii)]{NaikTripathi2026}.
\end{proof}

%%%%%%%%%%%%%%%%%%%%%%%%%%%%%%%%%%%%%%%%%%%%%%%%%%%%%%%%%%%%%
\section{Fully discrete approximation}\label{sec:fullydiscrete_approximation}
%%%%%%%%%%%%%%%%%%%%%%%%%%%%%%%%%%%%%%%%%%%%%%%%%%%%%%%%%%%%%

To numerically approximate \eqref{eq:SEE} via a standard Galerkin finite element method, let $\{V_h\}_{h \in (0,1]}$ be a quasi-uniform family of finite-dimensional subspaces of $\dot{H}^1$ consisting of continuous, piecewise linear functions defined over a regular triangulation of $\mathcal{O}$, with $h$ denoting the maximum mesh size. Following Kruse~\cite[Section~3]{MR3168284}, we define the generalized orthogonal projector $P_h \colon \dot{H}^{-1} \to V_h$ by 
\begin{equation}
    ( P_h x, y_h )_H = \langle x, y_h \rangle_{\dot{H}^{-1}, \dot{H}^1} \quad \text{for all } x \in \dot{H}^{-1}, y_h \in V_h,
\end{equation}
where $\langle \cdot, \cdot \rangle_{\dot{H}^{-1}, \dot{H}^1}$ denotes the duality pairing between $\dot{H}^{-1}$ and $\dot{H}^1$. Furthermore, let $A_h \colon V_h \to V_h$ be the discrete analogue of the operator $A$, defined by
\begin{equation}
    ( A_h x_h, y_h )_H = ( A^{1/2} x_h, A^{1/2} y_h )_H \quad \text{for all } x_h, y_h \in V_h.
\end{equation}
Since $A_h$ is self-adjoint and positive definite on the finite-dimensional space $V_h$, the operator $-A_h$ generates an analytic semigroup of contractions on $V_h$, denoted by $S_h(t) = e^{-t A_h}$. Projecting the continuous equation~\eqref{eq:SEE} onto the finite-dimensional subspace $V_h$ yields the following spatially semidiscrete initial value problem for the adapted process $X_h \colon [0,T]\times\Omega \to V_h$:
\begin{equation}\label{eq:SEE_semidiscrete}
\begin{aligned}
    dX_h(t) + A_h X_h(t)\,dt &= P_h F(t,X_h(t))\,dt + P_h G(t)\,dW(t), \quad t \in (0,T], \\
    X_h(0) &= P_h X_0.
\end{aligned}
\end{equation}

For the temporal discretization, we apply the linearly implicit Euler method to the semidiscrete problem~\eqref{eq:SEE_semidiscrete}. Let $N \in \mathbb{N}$ and introduce a uniform time step $k \coloneqq T/N$, defining the discrete time grid $t_n \coloneqq nk$ for $n=0,1,\dots,N$. The resulting fully discrete approximation $X_h^n \approx X(t_n)$ is generated by the recursion
\begin{equation}\label{eq:recursion}
    X_h^{n} = S_{h,k}X_h^{n-1} + k\,S_{h,k} P_h F(t_{n-1},X_h^{n-1}) + S_{h,k}P_h G(t_{n-1})\,\Delta W^{n}, \quad n=1,\dots,N,
\end{equation}
with the initial condition $X_h^0 = P_h X_0$, where $S_{h,k} \coloneqq (I + k A_h)^{-1}$ and $\Delta W^n \coloneqq W(t_{n}) - W(t_{n-1})$. 

The discrete operator $S_{h,k}$ satisfies a smoothing property (see, e.g., Kruse~\cite[Eq.~(4.8)]{MR3168284}): for any $\rho \in [0,1]$, there exists a constant $C > 0$ independent of $h, k$, and $n$ such that for all $n=1,\dots,N$,
\begin{equation}\label{eq:discrete_smoothing_estimate}
    \|A_h^{\rho} S_{h,k}^{n} x_h\|_H \le C\,t_{n}^{-\rho}\,\|x_h\|_H \quad \text{for all } x_h \in V_h.
\end{equation}
By an interpolation argument (cf.\ Kruse~\cite[Eqs.~(4.8) and (4.12)]{MR3168284}), this bound extends to negative norms evaluated at discrete time steps. For any $\eta \in [0, 1]$, there exists a constant $C>0$, independent of $h$, $k$, and $j$, such that
\begin{equation}\label{eq:negative_norm_smoothing_fully_discrete}
    \|S_{h,k}^j P_h x\|_H \le C t_j^{-\eta/2} \|x\|_{\dot{H}^{-\eta}} \quad \text{for all } x \in \dot{H}^{-\eta} \text{ and } j = 1, \dots, N.
\end{equation}
By iterative expansion of the recursion~\eqref{eq:recursion}, the fully discrete solution admits a discrete variation of constants formulation. Specifically, $X_h^n$ can be expressed explicitly as
\begin{equation}\label{eq:mild_fully_discrete}
    X_h^n = S_{h,k}^n\,X_h^0 
    + k \sum_{i=0}^{n-1} S_{h,k}^{\,n-i}\,P_h\,F(t_i,X_h^i) 
    + \sum_{i=0}^{n-1} S_{h,k}^{\,n-i}\,P_h\,G(t_i)\,\Delta W^{i+1}, \quad n=1,\dots,N.
\end{equation}

We state the main strong error estimate for the fully discrete scheme:

\begin{theorem}\label{thm:strong_convergence_of_fully_discrete_scheme}
Suppose that Assumptions~\ref{ass:initial_data}--\ref{ass:drift} hold. Let $p \in [2,\infty)$, $\beta \in (0,2]$, and $\mu \in (0,2]$, and define $\nu \coloneqq \min\{\beta, \mu\}$. Let $X(t)$ and $X_h^n$ be the mild and fully discrete solutions defined by \eqref{eq:mild_solution_SEE} and \eqref{eq:mild_fully_discrete}, respectively. Then, there exists a constant $C > 0$, independent of $h, k$, and $n$, such that for all $n \in \{1, 2, \dots, N\}$,
\begin{equation}\label{eq:fully_discrete_error_bound}
    \|X(t_n) - X_h^n\|_{L^p(\Omega;H)} \le C \Bigl(h^\beta + k^{\frac{\beta}{2}} \ell_n^{(\beta)}\Bigr) t_{n}^{-\frac{\beta-\nu}{2}} \Bigl(1 + \|X_{0}\|_{L^{2p}(\Omega;\dot{H}^{\nu})}^2\Bigr),
\end{equation}
where 
\begin{equation*}
    \ell_n^{(\beta)} \coloneqq 
    \begin{cases}
        1 & \text{if } \beta \in (0, 2), \\
        (1 + \ln(t_n/k))^{1/2} & \text{if } \beta = 2.
    \end{cases}
\end{equation*}
\end{theorem}

The proof of this theorem requires substantial preliminary estimates and is deferred to Section~\ref{sec:proof_of_fully_discrete_result}.

%%%%%%%%%%%%%%%%%%%%%%%%%%%%%%%%%%%%%%%%%%%%%%%%%%%%%%%%%%
\section{Regularity results}\label{sec:regularity_results}
%%%%%%%%%%%%%%%%%%%%%%%%%%%%%%%%%%%%%%%%%%%%%%%%%%%%%%%%%%

This section is devoted to the spatial and temporal regularity of the mild solution. The following lemmas are crucial for establishing the regularity results. The first lemma provides a Burkholder--Davis--Gundy (BDG) type inequality for $H$-valued stochastic integrals, which will be used repeatedly to estimate the moments of stochastic convolutions.

\begin{lemma}[{\cite[Proposition~2.12]{MR3154916}}]\label{lem:BDG}
Let \( p \ge 2 \) and \( 0 \le t_1 < t_2 \le T \).
Let \( \Phi \colon [0,T]\times\Omega \to {\mathrm{HS}(U_0, H)} \) be a predictable process such that
\begin{equation*}
    \mathbb{E}\left[ \left( \int_{t_1}^{t_2} \|\Phi(\sigma)\|_{\mathrm{HS}(U_0, H)}^2\, d\sigma \right)^{p/2} \right] < \infty .
\end{equation*}
Then the stochastic integral \( \int_{t_1}^{t_2} \Phi(\sigma)\, dW(\sigma) \) is well defined and satisfies
\begin{equation*}
    \mathbb{E}\left[ \left\| \int_{t_1}^{t_2} \Phi(\sigma)\, dW(\sigma) \right\|_H^p \right] \le C_p\, \mathbb{E}\left[ \left( \int_{t_1}^{t_2} \|\Phi(\sigma)\|_{\mathrm{HS}(U_0, H)}^2\, d\sigma \right)^{p/2} \right],
\end{equation*}
where the constant $C_p$ is given by
\begin{equation*}
    C_p = \left(\frac{p(p-1)}{2}\right)^{p/2} \left(\frac{p}{p-1}\right)^{p\left(\frac{p}{2}-1\right)} .
\end{equation*}
\end{lemma}

The following lemma collects several smoothing properties of the analytic $\mathcal{C}_0$-semigroup that will be essential throughout the article.

\begin{lemma}\label{lem:semigroup_estimates}
Let $\gamma^+ \coloneqq \max\{0,\gamma\}$ denote the positive part of $\gamma$, and let the operator $A$ and the corresponding analytic semigroup $\{S(t)\}_{t\ge 0}$ be as defined in Section~\ref{sec:introduction}. Then the following estimates hold:
\begin{enumerate}[label={\upshape(\roman*)}]
    \item For any $\gamma \in \mathbb{R}$, there exists a constant $C > 0$, depending on $\gamma$, such that
    \[
    \|A^{\gamma}S(t)\|_{\mathcal{L}(H)} \le C\, t^{-\gamma^+}, \quad t > 0.
    \]
    
    \item For any $\rho \in [0,1]$, there exists a constant $C > 0$, depending on $\rho$, such that
    \[
    \|A^{-\rho}(I-S(t))\|_{\mathcal{L}(H)} \le C\, t^{\rho}, \quad t > 0.
    \]
    
    \item For any $\rho \in [0,1]$, there exists a constant $C > 0$, depending on $\rho$, such that
    \[
    \int_{s}^{t} \|A^{\frac{\rho}{2}} S(t-\sigma)v\|_H^2 \,d\sigma \leq C \,(t - s)^{1-\rho} \|v\|_H^2 \quad \text{for all } v \in H, \; 0 \leq s < t.
    \]
    
    \item For any $\rho \in [0,1]$, there exists a constant $C > 0$, depending on $\rho$, such that
    \[
    \left\|A^\rho \int_{s}^{t} S(t-\sigma)v \, d\sigma \right\|_H \leq C \, (t - s)^{1-\rho} \|v\|_H \quad \text{for all } v \in H, \; 0 \leq s < t.
    \]   
\end{enumerate}
\end{lemma}

\begin{proof}
    The proof of (i) can be found in \cite[Lemma~2.1(i)]{NaikTripathi2026}, while the proofs for (ii)--(iv) are provided in \cite[Lemma~2.5]{MR3168284}.
\end{proof}

\begin{lemma} \label{lem:stochastic_convolution_regularity}
Let $\beta \in (0,2]$. Suppose Assumption~\ref{ass:diffusion} holds. Then for $0 \leq s < t \leq T$, we have
\[
\int_{s}^{t} \bigl\| A^{\frac{\alpha}{2}} S(t - \sigma) G(\sigma) \bigr\|_{\mathrm{HS}(U_0, H)}^2 \, d\sigma \leq C (t - s)^{\min\{\beta - \alpha, 1\}}, \quad \alpha \in [0, \beta].
\]
\end{lemma}

\begin{proof}
The proof follows the approach of \cite[Lemma~2.3]{MR3327065}. Adding and subtracting $G(t)$ yields
\begin{align}
\int_s^t \bigl\|A^{\frac{\alpha}{2}} S(t-\sigma) G(\sigma)\bigr\|_{\mathrm{HS}(U_0, H)}^2 \, d\sigma 
&\leq 2 \int_s^t \bigl\|A^{\frac{\alpha}{2}} S(t-\sigma)\bigl(G(\sigma)-G(t)\bigr)\bigr\|_{\mathrm{HS}(U_0, H)}^2 \, d\sigma \notag \\
&\quad + 2 \int_s^t \bigl\|A^{\frac{\alpha}{2}} S(t-\sigma) G(t)\bigr\|_{\mathrm{HS}(U_0, H)}^2 \, d\sigma \eqqcolon 2 (I_1 + I_2). \label{eq:stoc_conv_split}
\end{align}
We factor $A^{\frac{\alpha}{2}} = A^{\frac{\rho}{2}} A^{\frac{\beta-1}{2}}$ with $\rho \coloneqq 1+\alpha-\beta \le 1$, and note that $1 - \rho^+ = \min\{1-\rho, 1\} = \min\{\beta-\alpha, 1\} \ge 0$. For $I_1$, using $\|L M\|_{\mathrm{HS}(U_0, H)} \le \|L\|_{\mathcal{L}(H)} \|M\|_{\mathrm{HS}(U_0, H)}$, Lemma~\ref{lem:semigroup_estimates}(i), and Assumption~\ref{ass:diffusion} with $\delta > 0$, we obtain
\begin{equation}\label{eq:bound_I1}
    I_1 \le \int_s^t \bigl\|A^{\frac{\rho}{2}} S(t-\sigma)\bigr\|_{\mathcal{L}(H)}^2 \bigl\|G(\sigma)-G(t)\bigr\|_{\mathrm{HS}(U_0, \dot{H}^{\beta-1})}^2 \, d\sigma 
    \le C \int_s^t (t-\sigma)^{-\rho^+ + 2\delta} \, d\sigma \le C (t-s)^{\min\{\beta-\alpha, 1\} + 2\delta}.
\end{equation}
For $I_2$, expanding the Hilbert--Schmidt norm over an orthonormal basis $\{\psi_m\}_{m=1}^\infty$ of $U_0$ via Tonelli's theorem and applying Lemma~\ref{lem:semigroup_estimates}(iii) (for $\rho \in [0,1]$) or Lemma~\ref{lem:semigroup_estimates}(i) (for $\rho < 0$) together with Assumption~\ref{ass:diffusion} yields
\begin{align}
    I_2 &= \sum_{m=1}^\infty \int_s^t \bigl\| A^{\frac{\rho}{2}} S(t-\sigma) \bigl(A^{\frac{\beta-1}{2}} G(t) \psi_m\bigr) \bigr\|_H^2 \, d\sigma \notag \\
    &\le C (t-s)^{1-\rho^+} \sum_{m=1}^\infty \bigl\| A^{\frac{\beta-1}{2}} G(t) \psi_m \bigr\|_H^2 = C (t-s)^{\min\{\beta-\alpha, 1\}} \bigl\|G(t)\bigr\|_{\mathrm{HS}(U_0, \dot{H}^{\beta-1})}^2 \le C (t-s)^{\min\{\beta-\alpha, 1\}}. \label{eq:bound_I2_final}
\end{align}
Substituting \eqref{eq:bound_I1} and \eqref{eq:bound_I2_final} into \eqref{eq:stoc_conv_split} completes the proof.
\end{proof}

%%%%%%%%%%%%

In the following theorem, we establish weighted spatial and temporal regularity results for the mild solution. These estimates are essential for our subsequent numerical error analysis. 

\begin{theorem}\label{thm:singular_regularity}
Suppose Assumptions~\ref{ass:initial_data}, \ref{ass:diffusion}, and~\ref{ass:drift}{\upshape(i)}--{\upshape(ii)} hold, and let $\beta \in (0,2]$, $\mu \in (0,2]$, and $\nu \coloneqq \min\{\beta, \mu\}$. Then, for any moment exponent in $[2, 2p]$ (denoted below by $p$ for notational simplicity), the mild solution $X(t)$ satisfies the following regularity properties:
\begin{enumerate}[label={\upshape(\roman*)}]
    \item 
    There exists a constant $C > 0$ such that, for all $t \in (0,T]$,
    \begin{equation}\label{eq:weighted_spatial_regularity}
        \|X(t)\|_{L^{p}(\Omega;\dot H^{\beta})}
        \leq
            C \, t^{-\frac{\beta - \nu}{2}} \bigl( 1 + \|X_{0}\|_{L^{p}(\Omega;\dot H^{\nu})} \bigr).
    \end{equation}

    \item 
    There exists a constant $C > 0$ such that for $0 < s \le t \leq T$,
    \begin{equation}\label{eq:weighted_temporal_regularity}
        \|X(t) - X(s)\|_{L^p(\Omega;H)} 
        \leq 
            C \left( (t-s)^{\frac{\beta}{2}} \, s^{-\frac{\beta -\nu}{2}} + (t-s)^{\frac{\min\{\beta, 1\}}{2}} \right) \bigl(1 + \|X_{0}\|_{L^{p}(\Omega;\dot H^{\nu})}\bigr).
    \end{equation}

    \item 
    There exists a constant $C > 0$ such that for $0 < s \le t \leq T$,
    \begin{equation}\label{eq:weighted_temporal_regularity_combined}
        \|X(t) - X(s)\|_{L^p(\Omega;\dot{H}^{\max\{\beta-1, 0\}})} \leq C(t-s)^{\frac{\min\{\beta, 1\}}{2}} \, s^{-\frac{\beta -\nu}{2}} \bigl(1 + \|X_{0}\|_{L^{p}(\Omega;\dot H^{\nu})}\bigr).
    \end{equation}

    \item There exists a constant $C > 0$ such that for all $t \in [0,T]$,
    \begin{equation}\label{eq:initial_temporal_regularity}
        \|X(t) - X_0\|_{L^p(\Omega;H)} \le C t^{\frac{\min\{\nu, 1\}}{2}} \bigl(1 + \|X_{0}\|_{L^{p}(\Omega;\dot H^{\nu})}\bigr).
    \end{equation}

    \item If $\beta \in (1, 2]$ and $\nu > \beta-1$, there exists a constant $C > 0$ such that for all $t \in [0,T]$,
    \begin{equation}\label{eq:initial_temporal_regularity_higher}
        \|X(t) - X_0\|_{L^p(\Omega;\dot{H}^{\beta-1})} \le C t^{\frac{\nu - (\beta-1)}{2}} \bigl(1 + \|X_{0}\|_{L^{p}(\Omega;\dot H^{\nu})}\bigr).
    \end{equation}
\end{enumerate}
\end{theorem}

\begin{proof}
\noindent\textbf{Preliminary increment bounds.} For any $0 \le s \le t \le T$, the mild solution~\eqref{eq:mild_solution_SEE} admits the increment decomposition
\begin{align}
    X(t) - X(s) &= (S(t-s) - I) X(s) + \int_{s}^{t} S(t-\sigma)F(\sigma,X(\sigma))\,d\sigma + \int_{s}^{t} S(t-\sigma)G(\sigma)\,dW(\sigma) \notag \\
    &\eqqcolon I_1 + I_2 + I_3. \label{eq:temporal_split}
\end{align}
Let $r \in \{0, \max\{\beta-1, 0\}\}$. First, by Lemma~\ref{lem:BDG} and Lemma~\ref{lem:stochastic_convolution_regularity} (with $\alpha = r$, noting that $\min\{\beta-r, 1\} = \min\{\beta, 1\}$), the stochastic convolution satisfies, for all $0 \le s \le t \le T$,
\begin{equation}\label{eq:temp_bound_I3}
    \|I_3\|_{L^p(\Omega;\dot{H}^r)} \le C (t-s)^{\frac{\min\{\beta-r, 1\}}{2}} = C (t-s)^{\frac{\min\{\beta, 1\}}{2}}.
\end{equation}
Second, for the drift integral $I_2$, factoring $A^{r/2}S(t-\sigma) = A^{1/2}S(t-\sigma)A^{(r-1)/2}$ and applying Lemma~\ref{lem:semigroup_estimates}(i) together with Assumption~\ref{ass:drift}(i)--(ii) (\eqref{eq:F_linear_growth} for $r=0$ and \eqref{eq:F_growth_H_beta} for $r=\beta-1$) yields, for all $0 \le s \le t \le T$,
\begin{equation}\label{eq:temp_bound_I2_intermediate}
    \|I_2\|_{L^p(\Omega;\dot{H}^r)} \le C \int_s^t (t-\sigma)^{-1/2} \bigl(1 + \|X(\sigma)\|_{L^p(\Omega;\dot{H}^r)}\bigr) \, d\sigma.
\end{equation}

\medskip
\noindent(i) Taking the $L^{p}(\Omega;\dot H^{\beta})$-norm of \eqref{eq:mild_solution_SEE} and applying the triangle inequality gives
\begin{align}
\|X(t)\|_{L^p(\Omega;\dot{H}^\beta)} 
&\leq \|S(t)X_0\|_{L^p(\Omega;\dot{H}^\beta)} 
+ \left\| \int_0^t S(t-\sigma) F(\sigma,X(\sigma)) \,d\sigma\right\|_{L^p(\Omega;\dot{H}^\beta)} \notag \\
&\quad + \left\|\int_0^t S(t-\sigma) G(\sigma)\, dW(\sigma)\right\|_{L^p(\Omega;\dot{H}^\beta)} 
\eqqcolon J_1 + J_2 + J_3. \label{eq:spatial_split}
\end{align}
By Lemma~\ref{lem:semigroup_estimates}(i), and by Lemma~\ref{lem:BDG} together with Lemma~\ref{lem:stochastic_convolution_regularity} ($\alpha = \beta$), we have
\begin{equation}\label{eq:bound_J1}
    J_1 = \bigl\|A^{\frac{\beta-\nu}{2}}S(t) A^{\frac{\nu}{2}}X_0\bigr\|_{L^p(\Omega;H)}
    \leq C \, t^{-\frac{\beta-\nu}{2}} \|X_0\|_{L^p(\Omega;\dot{H}^\nu)},
\end{equation}
\begin{equation}\label{eq:bound_J3}
    J_3 \le C \biggl(\int_{0}^{t} \bigl\|A^{\frac{\beta}{2}}S(t-\sigma)G(\sigma)\bigr\|_{\mathrm{HS}(U_0, H)}^{2}\,d\sigma\biggr)^{\!1/2} \le C.
\end{equation}
To bound $J_2$, we consider three cases.

\medskip
\noindent \textbf{Case 1: $\beta \in (0,1)$.} Applying Lemma~\ref{lem:semigroup_estimates}(i), Assumption~\ref{ass:drift}(i), and Theorem~\ref{thm:base_regularity}(i) yields
\begin{align}
    J_2 &= \left\| \int_0^t A^{\frac{\beta+1}{2}} S(t-\sigma) A^{-1/2} F(\sigma,X(\sigma)) \,d\sigma\right\|_{L^p(\Omega;H)} \notag \\
    &\le C \int_0^t (t-\sigma)^{-\frac{\beta+1}{2}} \bigl(1 + \|X(\sigma)\|_{L^p(\Omega;H)}\bigr) \,d\sigma \le C \bigl(1 + \|X_0\|_{L^p(\Omega;H)}\bigr), \label{eq:bound_J2_case1}
\end{align}
which, combined with \eqref{eq:bound_J1} and \eqref{eq:bound_J3}, establishes \eqref{eq:weighted_spatial_regularity} for $\beta \in (0,1)$.

For $\beta \in [1,2]$, we isolate $F(t, X(t))$ to handle the singularity at $\sigma = t$ and split
\begin{align}
    J_2 &\le \left\| \int_0^t A^{\frac{\beta}{2}} S(t-\sigma) \bigl(F(\sigma,X(\sigma)) - F(t,X(t))\bigr) \,d\sigma\right\|_{L^p(\Omega;H)} \notag \\
    &\quad + \left\| \int_0^t A^{\frac{\beta}{2}} S(t-\sigma) F(t,X(t)) \,d\sigma\right\|_{L^p(\Omega;H)} \eqqcolon J_{2,a} + J_{2,b}. \label{eq:J2_split_case2}
\end{align}

\medskip
\noindent \textbf{Case 2: $\beta \in [1,2)$.} Factoring $A^{\frac{\beta}{2}}S(t-\sigma) = A^{\frac{\beta+1}{2}}S(t-\sigma)A^{-\frac{1}{2}}$ in $J_{2,a}$, substituting the mild solution representation of $X(t) - X(\sigma)$, and applying \eqref{eq:F_Lipschitz_H1} gives
\begin{align}
    J_{2,a} &\le C \int_0^t (t-\sigma)^{-\frac{\beta+1}{2}} \left( |t-\sigma|^{\beta/2} + \|X(t) - X(\sigma)\|_{L^p(\Omega;H)} \right) d\sigma \notag \\
    &\le C + C \int_0^t (t-\sigma)^{-\frac{\beta+1}{2}} \left( \bigl\|(S(t-\sigma)-I)S(\sigma)X_0\bigr\|_{L^p(\Omega;H)} + \|I_{F}\|_{L^p(\Omega;H)} + \|I_{G}\|_{L^p(\Omega;H)} \right) d\sigma, \label{eq:J2a_expanded}
\end{align}
where the convolution increments $I_F$ and $I_G$ are defined as
\begin{align*}
    I_F &\coloneqq \int_0^t S(t-\tau)F(\tau,X(\tau)) \, d\tau - \int_0^\sigma S(\sigma-\tau)F(\tau,X(\tau)) \, d\tau, \\
    I_G &\coloneqq \int_0^t S(t-\tau)G(\tau) \, dW(\tau) - \int_0^\sigma S(\sigma-\tau)G(\tau) \, dW(\tau).
\end{align*}
By \cite[Lemma~3.2]{NaikTripathi2026}, for any $\gamma \in (0, 1/2)$,
\begin{equation}\label{eq:increments_bound}
    \|I_F\|_{L^p(\Omega;H)} + \|I_G\|_{L^p(\Omega;H)} \le C\bigl(1 + \|X_0\|_{L^p(\Omega; H)}\bigr)(t-\sigma)^\gamma,
\end{equation}
while Lemma~\ref{lem:semigroup_estimates}(i)--(ii) (with $1 - \nu/2 > 0$) yield
\begin{equation}\label{eq:initial_data_bound}
    \bigl\|(S(t-\sigma)-I)S(\sigma)X_0\bigr\|_H 
    \le \bigl\|A^{-1}(S(t-\sigma)-I)\bigr\|_{\mathcal{L}(H)} \bigl\|A^{1-\frac{\nu}{2}} S(\sigma) A^{\frac{\nu}{2}} X_0\bigr\|_H 
    \le C (t-\sigma) \sigma^{-1+\nu/2} \|X_0\|_{\dot{H}^\nu}.
\end{equation}
Choosing $\gamma \in (\frac{\beta-1}{2}, \frac{1}{2})$ so that $-\frac{\beta+1}{2}+\gamma > -1$, and substituting \eqref{eq:increments_bound}--\eqref{eq:initial_data_bound} into \eqref{eq:J2a_expanded}, we obtain
\begin{align}
    J_{2,a} &\le C + C \|X_0\|_{L^p(\Omega;\dot{H}^\nu)} \int_0^t (t-\sigma)^{\frac{1-\beta}{2}} \sigma^{-1+\nu/2} \, d\sigma 
    + C \bigl(1 + \|X_0\|_{L^p(\Omega; H)}\bigr) \int_0^t (t-\sigma)^{-\frac{\beta+1}{2}+\gamma} \, d\sigma \notag \\
    &\le C \Bigl(1 + t^{\frac{1-\beta+\nu}{2}} B\bigl(\tfrac{3-\beta}{2}, \tfrac{\nu}{2}\bigr)\Bigr) \bigl(1 + \|X_0\|_{L^p(\Omega;\dot{H}^\nu)}\bigr) 
    \le C \, t^{-\frac{\beta-\nu}{2}} \bigl(1 + \|X_0\|_{L^p(\Omega;\dot{H}^\nu)}\bigr). \label{eq:bound_J2a}
\end{align}
For $J_{2,b}$, let $\nu' \coloneqq \min\{\beta-1, \mu\}$. Since $\beta - 1 \in [0, 1)$, applying Theorem~\ref{thm:base_regularity}(i) (when $\beta = 1$) and Case~1 to the exponent $\beta - 1 \in (0,1)$ (when $\beta \in (1,2)$) yields
\begin{align}
    \|X(t)\|_{L^p(\Omega;\dot{H}^{\beta-1})} 
    &\le C t^{-\frac{\beta-1-\nu'}{2}} \bigl(1 + \|X_0\|_{L^p(\Omega;\dot{H}^{\nu'})}\bigr) \label{eq:X_H_beta_minus_1_sharp} \\
    &\le C t^{-\frac{\beta-\nu}{2}} \bigl(1 + \|X_0\|_{L^p(\Omega;\dot{H}^\nu)}\bigr), \label{eq:X_H_beta_minus_1_relaxed}
\end{align}
where we used $\nu' \le \nu \le \nu' + 1$ to bound $\|X_0\|_{\dot{H}^{\nu'}} \le C\|X_0\|_{\dot{H}^\nu}$ and $t^{-\frac{\beta-1-\nu'}{2}} \le C t^{-\frac{\beta-\nu}{2}}$ on $(0,T]$. Using $\int_0^t S(t-\sigma) v \, d\sigma = A^{-1}(I - S(t))v$, Lemma~\ref{lem:semigroup_estimates}(ii), Assumption~\ref{ass:drift}(i)--(ii), and \eqref{eq:X_H_beta_minus_1_relaxed}, we obtain
\begin{align}
    J_{2,b} &= \bigl\| A^{\frac{\beta-2}{2}} (I - S(t)) F(t, X(t)) \bigr\|_{L^p(\Omega;H)} 
    \le C \bigl\|F(t,X(t))\bigr\|_{L^p(\Omega;\dot{H}^{\beta-2})} \label{eq:J2b_intermediate} \\
    &\le C \bigl(1 + \|X(t)\|_{L^p(\Omega;\dot{H}^{\beta-1})}\bigr) 
    \le C t^{-\frac{\beta-\nu}{2}} \bigl(1 + \|X_0\|_{L^p(\Omega;\dot{H}^\nu)}\bigr). \label{eq:bound_J2b_betagt1}
\end{align}
Combining \eqref{eq:bound_J1}, \eqref{eq:bound_J3}, \eqref{eq:bound_J2a}, and \eqref{eq:bound_J2b_betagt1} establishes \eqref{eq:weighted_spatial_regularity} for all $\beta \in [1,2)$.

\medskip
\noindent \textbf{Case 3: $\beta = 2$.} Here $\nu = \min\{2, \mu\} = \mu$ and $\nu' \coloneqq \min\{1, \mu\}$. Because Assumptions~\ref{ass:diffusion} and~\ref{ass:drift}{\upshape(i)} with $\beta = 2$ imply the corresponding hypotheses for $\beta = 1$, applying the $\beta = 1$ case of \eqref{eq:weighted_spatial_regularity} already established in Case~2 yields, for all $\tau \in (0,T]$,
\begin{equation}\label{eq:X_H1_bound_beta2}
    \|X(\tau)\|_{L^p(\Omega;\dot{H}^1)} \le C \tau^{-\frac{1-\nu'}{2}} \bigl(1 + \|X_0\|_{L^p(\Omega;\dot{H}^{\nu'})}\bigr) \le C \tau^{-\frac{1-\nu'}{2}} \bigl(1 + \|X_0\|_{L^p(\Omega;\dot{H}^\nu)}\bigr),
\end{equation}
where we used $\nu' \le \nu$ so that $\dot{H}^\nu \hookrightarrow \dot{H}^{\nu'}$. Combining \eqref{eq:X_H1_bound_beta2} with Assumption~\ref{ass:drift}{\upshape(ii)} at $\beta = 2$ (namely, $\|F(\tau, v)\|_H \le C(1 + \|v\|_{\dot{H}^1})$ from \eqref{eq:F_growth_H_beta}) and noting that $1 \le T^{\frac{1-\nu'}{2}} \tau^{-\frac{1-\nu'}{2}}$, we obtain for all $\tau \in (0,T]$,
\begin{equation}\label{eq:F_H_bound_beta2}
    \|F(\tau, X(\tau))\|_{L^p(\Omega;H)} \le C\bigl(1 + \|X(\tau)\|_{L^p(\Omega;\dot{H}^1)}\bigr) \le C\tau^{-\frac{1-\nu'}{2}}\bigl(1 + \|X_0\|_{L^p(\Omega;\dot{H}^\nu)}\bigr).
\end{equation}
Using $\int_0^t A S(t-\sigma) v \, d\sigma = (I - S(t))v$, and \eqref{eq:F_H_bound_beta2} at $\tau = t$, we bound $J_{2,b}$ at $\beta = 2$:
\begin{equation}\label{eq:bound_J2b_beta2}
    J_{2,b} = \bigl\| (I - S(t)) F(t, X(t)) \bigr\|_{L^p(\Omega;H)} \le C \|F(t, X(t))\|_{L^p(\Omega;H)} \le C t^{-\frac{2-\nu}{2}} \bigl(1 + \|X_0\|_{L^p(\Omega;\dot{H}^\nu)}\bigr),
\end{equation}
where the last step uses $1 - \nu' \le 2 - \nu$ so that $t^{-\frac{1-\nu'}{2}} \le C t^{-\frac{2-\nu}{2}}$ on $(0,T]$.

To bound $J_{2,a}$ at $\beta = 2$, we first establish an intermediate spatial regularity bound in $\dot{H}^{3/2}$. Taking the $L^p(\Omega;\dot{H}^{3/2})$-norm in \eqref{eq:mild_solution_SEE} and applying Lemma~\ref{lem:semigroup_estimates}{\upshape(i)}, \eqref{eq:F_H_bound_beta2}, and Lemma~\ref{lem:BDG} together with Lemma~\ref{lem:stochastic_convolution_regularity} ($\alpha = 3/2$) yields, for all $\sigma \in (0,T]$,
\begin{align}
    \|X(\sigma)\|_{L^p(\Omega;\dot{H}^{3/2})} &\le \bigl\|A^{\frac{3/2-\nu'}{2}} S(\sigma) A^{\frac{\nu'}{2}} X_0\bigr\|_{L^p(\Omega;H)} + \int_0^\sigma \bigl\|A^{3/4} S(\sigma-\tau)\bigr\|_{\mathcal{L}(H)} \|F(\tau, X(\tau))\|_{L^p(\Omega;H)} \, d\tau \notag \\
    &\quad + \left\| \int_0^\sigma A^{3/4} S(\sigma-\tau) G(\tau) \, dW(\tau) \right\|_{L^p(\Omega;H)} \notag \\
    &\le C\sigma^{-\frac{3/2-\nu'}{2}}\|X_0\|_{L^p(\Omega;\dot{H}^{\nu'})} + C\bigl(1+\|X_0\|_{L^p(\Omega;\dot{H}^\nu)}\bigr)\int_0^\sigma (\sigma-\tau)^{-3/4}\tau^{-\frac{1-\nu'}{2}}\,d\tau + C \notag \\
    &\le C\sigma^{-\frac{3/2-\nu'}{2}}\bigl(1+\|X_0\|_{L^p(\Omega;\dot{H}^\nu)}\bigr), \label{eq:X_H_3half_bound}
\end{align}
where the Beta integral converges since $\frac{1-\nu'}{2} < \frac{1}{2} < 1$ and evaluates to $B\bigl(\tfrac{1}{4}, \tfrac{1+\nu'}{2}\bigr)\sigma^{-\frac{1/2-\nu'}{2}} \le C \sigma^{-\frac{3/2-\nu'}{2}}$. Next, for $0 < \sigma \le t \le T$, we estimate the temporal increment $X(t) - X(\sigma)$ in $L^p(\Omega;\dot{H}^1)$ via the decomposition \eqref{eq:temporal_split}. Applying Lemma~\ref{lem:semigroup_estimates}{\upshape(ii)} ($\rho = 1/4$) and \eqref{eq:X_H_3half_bound} to the initial term, Lemma~\ref{lem:semigroup_estimates}{\upshape(i)} and \eqref{eq:F_H_bound_beta2} (with $\tau \ge \sigma$) to the drift integral, and \eqref{eq:temp_bound_I3} ($r = 1$) to the stochastic convolution gives
\begin{align}
    \|X(t) - X(\sigma)\|_{L^p(\Omega;\dot{H}^1)} &\le \bigl\|A^{-1/4}(S(t-\sigma) - I)\bigr\|_{\mathcal{L}(H)} \|X(\sigma)\|_{L^p(\Omega;\dot{H}^{3/2})} \notag \\
    &\quad + \int_\sigma^t \bigl\|A^{1/2} S(t-\tau)\bigr\|_{\mathcal{L}(H)} \|F(\tau, X(\tau))\|_{L^p(\Omega;H)} \, d\tau + \|I_3\|_{L^p(\Omega;\dot{H}^1)} \notag \\
    &\le C\Bigl( (t-\sigma)^{1/4}\sigma^{-\frac{3/2-\nu'}{2}} + (t-\sigma)^{1/2}\sigma^{-\frac{1-\nu'}{2}} + (t-\sigma)^{1/2} \Bigr)\bigl(1+\|X_0\|_{L^p(\Omega;\dot{H}^\nu)}\bigr) \notag \\
    &\le C(t-\sigma)^{1/4}\sigma^{-\frac{3/2-\nu'}{2}}\bigl(1+\|X_0\|_{L^p(\Omega;\dot{H}^\nu)}\bigr). \label{eq:X_H1_increment_beta2}
\end{align}
Finally, applying Lemma~\ref{lem:semigroup_estimates}{\upshape(i)} ($\gamma=1$), Assumption~\ref{ass:drift}{\upshape(ii)} (\eqref{eq:F_Lipschitz_H_beta} at $\beta = 2$), and \eqref{eq:X_H1_increment_beta2} yields
\begin{align}
    J_{2,a} &= \left\| \int_0^t A S(t-\sigma)\bigl(F(\sigma, X(\sigma)) - F(t, X(t))\bigr)\,d\sigma \right\|_{L^p(\Omega;H)} \notag \\
    &\le C \int_0^t (t-\sigma)^{-1} \Bigl( (t-\sigma) + \|X(t) - X(\sigma)\|_{L^p(\Omega;\dot{H}^1)} \Bigr)\,d\sigma \notag \\
    &\le C\bigl(1 + \|X_0\|_{L^p(\Omega;\dot{H}^\nu)}\bigr) \int_0^t \Bigl( 1 + (t-\sigma)^{-3/4}\sigma^{-\frac{3/2-\nu'}{2}} \Bigr)\,d\sigma \notag \\
    &\le C\Bigl( t + B\bigl(\tfrac{1}{4}, \tfrac{1+2\nu'}{4}\bigr) t^{-\frac{1-\nu'}{2}} \Bigr)\bigl(1 + \|X_0\|_{L^p(\Omega;\dot{H}^\nu)}\bigr) \notag \\
    &\le C t^{-\frac{2-\nu}{2}}\bigl(1 + \|X_0\|_{L^p(\Omega;\dot{H}^\nu)}\bigr). \label{eq:bound_J2a_beta2}
\end{align}
Combining \eqref{eq:bound_J1}, \eqref{eq:bound_J3}, \eqref{eq:bound_J2b_beta2}, and \eqref{eq:bound_J2a_beta2} establishes \eqref{eq:weighted_spatial_regularity} for $\beta = 2$, completing the proof of~{\upshape(i)}.

\medskip
\noindent(ii) Let $0 < s \le t \le T$ and set $r = 0$. For $I_1$, applying Lemma~\ref{lem:semigroup_estimates}(ii) and the spatial regularity bound \eqref{eq:weighted_spatial_regularity} from part~(i)  gives
\begin{equation}\label{eq:temp_bound_I1_r0}
    \|I_1\|_{L^p(\Omega;H)} 
    \le \bigl\| A^{-\frac{\beta}{2}} (S(t-s) - I) \bigr\|_{\mathcal{L}(H)} \|X(s)\|_{L^p(\Omega;\dot{H}^\beta)} 
    \le C \, (t-s)^{\frac{\beta}{2}} \, s^{-\frac{\beta -\nu}{2}}\bigl(1 + \|X_{0}\|_{L^{p}(\Omega;\dot H^{\nu})}\bigr).
\end{equation}
For $I_2$, Theorem~\ref{thm:base_regularity}(i) and $\dot{H}^\nu \hookrightarrow H$ ensure $\sup_{\sigma \in [0,T]} \|X(\sigma)\|_{L^p(\Omega;H)} \le C(1+\|X_0\|_{L^p(\Omega;\dot{H}^\nu)})$. Substituting this uniform bound into \eqref{eq:temp_bound_I2_intermediate} yields, for all $0 \le s \le t \le T$,
\begin{equation}\label{eq:temp_bound_I2}
    \|I_2\|_{L^p(\Omega;H)} \le C (t-s)^{1/2} \bigl(1 + \|X_0\|_{L^p(\Omega;\dot{H}^\nu)}\bigr) \le C (t-s)^{\frac{\min\{\beta, 1\}}{2}} \bigl(1 + \|X_0\|_{L^p(\Omega;\dot{H}^\nu)}\bigr).
\end{equation}
Combining \eqref{eq:temp_bound_I1_r0}, \eqref{eq:temp_bound_I2}, and \eqref{eq:temp_bound_I3} (at $r=0$) gives \eqref{eq:weighted_temporal_regularity}. Moreover, since $(t-s)^{\beta/2} \le C(t-s)^{\min\{\beta,1\}/2}$ and $1 \le C s^{-(\beta-\nu)/2}$ on $(0,T]$, we obtain the consolidated bound
\begin{equation}\label{eq:consolidated_temporal_bound}
    \|X(t) - X(s)\|_{L^p(\Omega;H)} \leq C \, s^{-\frac{\beta -\nu}{2}} (t-s)^{\frac{\min\{\beta, 1\}}{2}} \bigl(1 + \|X_{0}\|_{L^{p}(\Omega;\dot H^{\nu})}\bigr).
\end{equation}

\medskip
\noindent(iii) Let $0 < s \le t \le T$ and set $r = \max\{\beta-1, 0\}$. When $\beta \in (0,1]$, we have $r=0$ and \eqref{eq:weighted_temporal_regularity_combined} reduces directly to \eqref{eq:consolidated_temporal_bound}. When $\beta \in (1,2]$, we have $r = \beta-1$ and $\min\{\beta,1\} = 1$. For $I_1$, applying Lemma~\ref{lem:semigroup_estimates}(ii) and \eqref{eq:weighted_spatial_regularity} gives
\begin{equation}\label{eq:temp_bound_I1}
    \|I_1\|_{L^p(\Omega;\dot{H}^{\beta-1})} 
    \le \bigl\| A^{-\frac{1}{2}} (S(t-s) - I) \bigr\|_{\mathcal{L}(H)} \|X(s)\|_{L^p(\Omega;\dot{H}^\beta)} 
    \le C \, (t-s)^{1/2} \, s^{-\frac{\beta -\nu}{2}}\bigl(1 + \|X_{0}\|_{L^{p}(\Omega;\dot H^{\nu})}\bigr).
\end{equation}
For $I_2$, using \eqref{eq:X_H_beta_minus_1_relaxed} and $\sigma \ge s$ to bound $\|X(\sigma)\|_{L^p(\Omega;\dot{H}^{\beta-1})} \le C s^{-\frac{\beta-\nu}{2}}(1+\|X_0\|_{L^p(\Omega;\dot{H}^\nu)})$ on $[s,t]$, \eqref{eq:temp_bound_I2_intermediate} integrates to
\begin{equation}\label{eq:temp_bound_I2_singular}
    \|I_2\|_{L^p(\Omega;\dot{H}^{\beta-1})} \le C (t-s)^{1/2} s^{-\frac{\beta-\nu}{2}} \bigl(1 + \|X_0\|_{L^p(\Omega;\dot{H}^\nu)}\bigr).
\end{equation}
Combining \eqref{eq:temp_bound_I1}, \eqref{eq:temp_bound_I2_singular}, and \eqref{eq:temp_bound_I3} (using $1 \le C s^{-(\beta-\nu)/2}$) yields \eqref{eq:weighted_temporal_regularity_combined}.

\medskip
\noindent(iv) Set $r = 0$ and $s = 0$, so that $I_1 = (S(t) - I)X_0$. Applying Lemma~\ref{lem:semigroup_estimates}(ii) with $\rho = \frac{\min\{\nu, 1\}}{2} \in (0, 1/2]$ and the embedding $\dot{H}^\nu \hookrightarrow \dot{H}^{\min\{\nu, 1\}}$ gives
\[
    \|I_1\|_{L^p(\Omega;H)} = \bigl\| A^{-\frac{\min\{\nu, 1\}}{2}} (S(t) - I) A^{\frac{\min\{\nu, 1\}}{2}} X_0 \bigr\|_{L^p(\Omega;H)} \le C t^{\frac{\min\{\nu, 1\}}{2}} \|X_0\|_{L^p(\Omega;\dot{H}^\nu)}.
\]
Combining this with \eqref{eq:temp_bound_I2} and \eqref{eq:temp_bound_I3} at $s=0$ (and noting $\min\{\nu, 1\} \le \min\{\beta, 1\}$) yields \eqref{eq:initial_temporal_regularity}.

\medskip
\noindent(v) Suppose $\beta \in (1,2]$ and $\nu > \beta-1$, and set $r = \beta-1$ and $s = 0$. Since $0 < \nu - (\beta-1) \le \beta - (\beta-1) = 1$, applying Lemma~\ref{lem:semigroup_estimates}(ii) with $\rho = \frac{\nu - (\beta-1)}{2} \in (0, 1/2]$ gives
\[
    \|I_1\|_{L^p(\Omega;\dot{H}^{\beta-1})} = \bigl\| A^{-\frac{\nu - (\beta-1)}{2}} (S(t) - I) A^{\frac{\nu}{2}} X_0 \bigr\|_{L^p(\Omega;H)} \le C t^{\frac{\nu - (\beta-1)}{2}} \|X_0\|_{L^p(\Omega;\dot{H}^\nu)}.
\]
For $I_2$, since $\nu > \beta-1$ implies $\nu' = \min\{\beta-1, \mu\} = \beta-1$, the singular time weight in \eqref{eq:X_H_beta_minus_1_sharp} vanishes, giving $\sup_{\sigma \in [0,T]} \|X(\sigma)\|_{L^p(\Omega;\dot{H}^{\beta-1})} \le C(1+\|X_0\|_{L^p(\Omega;\dot{H}^\nu)})$. Substituting this into \eqref{eq:temp_bound_I2_intermediate} at $s=0$ yields $\|I_2\|_{L^p(\Omega;\dot{H}^{\beta-1})} \le C t^{1/2}(1+\|X_0\|_{L^p(\Omega;\dot{H}^\nu)})$. Combining the bounds for $I_1, I_2$, and \eqref{eq:temp_bound_I3} (since $\nu - (\beta-1) \le 1 = \min\{\beta, 1\}$) establishes \eqref{eq:initial_temporal_regularity_higher}.
\end{proof}

%%%%%%%%%%%%%%%%%%%%%%%%%%%%%%%%%%%%%%%%%%%%%%%%%%%%%%%%%%%%%
\section{Proof of Theorem~\ref{thm:strong_convergence_of_fully_discrete_scheme}}\label{sec:proof_of_fully_discrete_result}
%%%%%%%%%%%%%%%%%%%%%%%%%%%%%%%%%%%%%%%%%%%%%%%%%%%%%%%%%%%%%

Before proceeding to the proof, we first collect several essential error bounds for the corresponding deterministic linear problem:
\begin{equation}\label{eq:deterministic}
    \frac{d}{dt} u(t) + Au(t) = 0, 
    \quad u(0) = v, \quad t\in(0,T].
\end{equation}
With the operators $P_h$, $A_h$, and $S_{h,k}$ defined in Section~\ref{sec:fullydiscrete_approximation}, the following lemma collects error estimates by comparing the exact solution $u(t) = S(t)v$ with the fully discrete solution $U_h^n = S_{h,k}^n P_h v$ generated via the linearly implicit Euler method.

\begin{lemma}\label{lem:fully_discrete_error_estimates}
The following estimates hold for the fully discrete approximation for all $t_n > 0$ and $h, k \in (0,1]$:
\begin{enumerate}[label={\upshape(\roman*)}]
    \item Let $0 \leq \rho \leq 2$ and $-\min\{1,2-\rho\} \leq \eta \leq \rho$. Then, there exists a constant $C>0$ such that
    \[
        \bigl\|\bigl(S(t_{n})-S_{h,k}^{n}P_{h}\bigr)v\bigr\|_H \le C\bigl(h^{\rho}+k^{\frac{\rho}{2}}\bigr)\,t_n^{-\frac{\rho - \eta}{2}}\,\|v\|_{\dot H^{\eta}} \quad \text{for all } v \in \dot{H}^{\eta}.
    \]
    \item For any $0 \le \rho \le 1$, there exists a constant $C>0$ such that 
    \[
        \left\|\sum_{j=1}^{n}\int_{t_{j-1}}^{t_{j}} \bigl(S(\sigma) - S_{h,k}^j P_h\bigr)v\,d\sigma\right\|_H \le C\bigl(h^{2-\rho}+k^{\frac{2-\rho}{2}}\bigr)\,\|v\|_{\dot H^{-\rho}} \quad \text{for all } v \in \dot{H}^{-\rho}.
    \]
    \item For any $0 \le \rho \le 2$, there exists a constant $C>0$ such that 
    \[
        \left(\sum_{j=1}^{n}\int_{t_{j-1}}^{t_{j}} \bigl(S(\sigma) - S_{h,k}^j P_h\bigr)v\bigr\|_H^{2}\,d\sigma\right)^{\!1/2} \le C\bigl(h^{\rho}+k^{\frac{\rho}{2}}\bigr)\,\|v\|_{\dot H^{\rho-1}} \quad \text{for all } v \in \dot{H}^{\rho-1}.
    \]
\end{enumerate}
\end{lemma}

\begin{proof}
The error bound in (i) is a direct consequence of the deterministic error estimate established by Andersson et al.~\cite[Lemma~5.1]{MR3475837}. The bounds in (ii) and (iii) are derived in Wang~\cite[Lemma~3.2]{MR3649432}.
\end{proof}

In addition, to establish the strong error estimate for the fully discrete scheme, we require two auxiliary results: a discrete Burkholder--Davis--Gundy (BDG) inequality and a generalized discrete fractional Grönwall lemma. The latter follows the formulation of Andersson et al.~\cite[Lemma~2.1]{MR3475837}, which is based on the foundational inequality proved by Elliott and Larsson~\cite[Lemma~7.1]{MR1122067}.

\begin{lemma}[{\cite[Lemma 4.2]{MR3649432}}]\label{lem:discrete_BDG_inequality}
Let $Z_m \colon \Omega \to H$, for $m=1,2,\dots,M$, be measurable mappings from $(\Omega,\mathcal{F})$ to $(H,\mathcal{B}(H))$ such that $\mathbb{E}[\|Z_{m}\|_H^{p}]<\infty$. Assume that
\begin{equation}\label{eq:martingale_difference_sequence}
\mathbb{E}\bigl[Z_{m+1}\,\big|\,Z_{1},\dots,Z_{m}\bigr]=0
\quad\text{for }m=1,\dots,M-1.
\end{equation}
Then for any $p\in[2,\infty)$, there exists a constant $C_{p}>0$ independent of $m$ such that
\begin{equation}\label{eq:discrete_BDG}
\left\|\sum_{i=1}^{m}Z_{i}\right\|_{L^{p}(\Omega;H)}
\le C_{p} \left(\sum_{i=1}^{m}\|Z_{i}\|_{L^{p}(\Omega;H)}^{2}\right)^{\!1/2}
\quad\text{for }m=1,\dots, M.
\end{equation}
\end{lemma}

\begin{lemma}[{\cite[Lemma~2.1]{MR3475837}}]\label{lem:discrete_gronwall}
Let $T > 0$, $N \in \mathbb{N}$, $k = T/N$, and $t_n = n k$ for $0 \le n \le N$. If $\{\varphi_n\}_{n=1}^N$ is a sequence of non-negative real numbers satisfying
\[
    \varphi_n \le C_1 \left(1 + t_n^{-1+\theta}\right) + C_2 k \sum_{j=0}^{n-1} t_{n-j}^{-1+\alpha} \varphi_j, \quad 1 \le n \le N,
\]
for some constants $C_1, C_2 \ge 0$ and $\theta, \alpha > 0$, then there exists a constant $C = C(\theta, \alpha, C_2, T) > 0$ such that
\[
    \varphi_n \le C C_1 \left(1 + t_n^{-1+\theta}\right), \quad 1 \le n \le N.
\]
\end{lemma}

We also need the following lemma.

\begin{lemma}\label{lem:operator_splitting}
Let $0 \leq \rho \leq 2$ and $-\min\{1, 2-\rho\} \le \eta \le \rho$. Then, for any $0 \le i \le n-1$ and $\sigma \in [t_i, t_{i+1})$, there exists a constant $C > 0$ such that
\begin{equation}\label{eq:operator_splitting_bound}
    \bigl\| \bigl(S(t_n-\sigma) - S_{h,k}^{n-i}P_h\bigr) A^{-\frac{\eta}{2}} \bigr\|_{\mathcal{L}(H)} \le C \bigl( h^\rho + k^{\rho/2} \bigr) (t_n-\sigma)^{-\frac{\rho - \eta}{2}}.
\end{equation}
\end{lemma}

\begin{proof}
For $\sigma \in [t_i, t_{i+1})$, adding and subtracting $S(t_{n-i}) = S(t_n-\sigma)S(\sigma-t_i)$ and applying Lemmas~\ref{lem:semigroup_estimates}(i)--(ii) and \ref{lem:fully_discrete_error_estimates}(i) (since $\sigma - t_i \le k$ and $t_{n-i} \ge t_n - \sigma$) yields
\begin{align}
    &\bigl\| \bigl(S(t_n-\sigma) - S_{h,k}^{n-i}P_h\bigr) A^{-\frac{\eta}{2}} \bigr\|_{\mathcal{L}(H)} \notag \\
    &\quad\le \bigl\| S(t_n-\sigma) A^{\frac{\rho - \eta}{2}} \bigl(I - S(\sigma - t_i)\bigr) A^{-\frac{\rho}{2}} \bigr\|_{\mathcal{L}(H)} 
    + \bigl\| \bigl(S(t_{n-i}) - S_{h,k}^{n-i}P_h\bigr) A^{-\frac{\eta}{2}} \bigr\|_{\mathcal{L}(H)} \notag \\
    &\quad\le C (t_n-\sigma)^{-\frac{\rho - \eta}{2}} (\sigma - t_i)^{\frac{\rho}{2}} + C \bigl( h^\rho + k^{\frac{\rho}{2}} \bigr) t_{n-i}^{-\frac{\rho - \eta}{2}} 
    \le C \bigl( h^\rho + k^{\frac{\rho}{2}} \bigr) (t_n-\sigma)^{-\frac{\rho - \eta}{2}}. \label{eq:operator_splitting_final}
\end{align}
\end{proof}

\begin{proof}[\textbf{Proof of Theorem~\ref{thm:strong_convergence_of_fully_discrete_scheme}}]
From \eqref{eq:mild_solution_SEE} and \eqref{eq:mild_fully_discrete}, applying the triangle inequality yields
\begin{align}
\|X(t_{n}) - X_h^{n}\|_{L^{p}(\Omega;H)}
&\le \left\|S(t_{n})X_{0} - S_{h,k}^{n}P_h X_{0}\right\|_{L^{p}(\Omega;H)} \notag \\
&\quad+ \left\|\sum_{i=0}^{n-1}\int_{t_{i}}^{t_{i+1}}
\left( S(t_{n}-\sigma)F(\sigma,X(\sigma))
- S_{h,k}^{n-i}P_{h}F(t_i,X_h^i)\right) \, d\sigma \right\|_{L^{p}(\Omega;H)} \notag \\
&\quad+ \left\|\sum_{i=0}^{n-1}\int_{t_{i}}^{t_{i+1}}
\left( S(t_{n}-\sigma)G(\sigma)
- S_{h,k}^{n-i}P_{h}G(t_i)\right) \, dW(\sigma) \right\|_{L^{p}(\Omega;H)} \notag \\
&\eqqcolon J_{0} + J_{F} + J_{G}. \label{eq:bound_fullydiscrete}
\end{align}
For $J_0$, Lemma~\ref{lem:fully_discrete_error_estimates}(i) yields
\begin{equation}\label{eq:bound_J0}
J_{0}
= \left\|\bigl(S(t_{n}) - S_{h,k}^{n}P_{h}\bigr)X_{0}\right\|_{L^{p}(\Omega;H)} \leq C\bigl(h^{\beta}+k^{\frac{\beta}{2}}\bigr)\,t_{n}^{-\frac{\beta -\nu}{2}}\,\|X_{0}\|_{L^{p}(\Omega;\dot H^{\nu})}.
\end{equation}
We decompose the drift error $J_F$ into three terms by adding and subtracting intermediate evaluations:
\begin{align}
J_{F}
&\leq
\left\|\sum_{i=0}^{n-1}\int_{t_{i}}^{t_{i+1}}
S(t_{n}-\sigma)\bigl(F(\sigma,X(\sigma)) - F(t_{i},X(t_{i}))\bigr)\,d\sigma\right\|_{L^{p}(\Omega;H)} \notag \\
&\quad+\left\|\sum_{i=0}^{n-1}\int_{t_{i}}^{t_{i+1}}
\bigl(S(t_{n}-\sigma) - S_{h,k}^{n-i}P_{h}\bigr)F(t_{i},X(t_{i}))\,d\sigma\right\|_{L^{p}(\Omega;H)} \notag \\
&\quad+\left\|\sum_{i=0}^{n-1}\int_{t_{i}}^{t_{i+1}}
S_{h,k}^{n-i}P_{h}\bigl(F(t_{i},X(t_{i})) - F(t_i,X_h^{i})\bigr)\,d\sigma\right\|_{L^{p}(\Omega;H)} \eqqcolon J_{F1} + J_{F2} + J_{F3}. \label{eq:bound_JF}
\end{align}

\noindent\textbf{Estimate of $J_{F3}$:} Applying \eqref{eq:negative_norm_smoothing_fully_discrete} with $\eta = 1$ and Assumption~\ref{ass:drift}{\upshape(i)} yields
\begin{align}
J_{F3} 
&\leq C \sum_{i=0}^{n-1}\int_{t_{i}}^{t_{i+1}} 
t_{n-i}^{-1/2} \bigl\|F(t_{i},X(t_{i})) - F(t_i,X_h^{i})\bigr\|_{L^{p}(\Omega;\dot{H}^{-1})}\,d\sigma 
\le C k \sum_{i=0}^{n-1} t_{n-i}^{-1/2} \|X(t_{i})-X_h^{i}\|_{L^{p}(\Omega;H)} \notag \\
&\le C\bigl(h^\beta + k^{\frac{\beta}{2}}\bigr) t_n^{-\frac{\beta-\nu}{2}} \|X_0\|_{L^p(\Omega;\dot{H}^\nu)} + C k \sum_{i=1}^{n-1} t_{n-i}^{-1/2} \|X(t_{i})-X_h^{i}\|_{L^{p}(\Omega;H)}, \label{eq:bound_JF3}
\end{align}
where for $i = 0$ we used $\|(I - P_h)X_0\|_H \le C h^\nu \|X_0\|_{\dot{H}^\nu}$, Young's inequality $h^\nu k^{\frac{\beta-\nu}{2}} \le h^\beta + k^{\frac{\beta}{2}}$, and $k^{\frac{2-\beta+\nu}{2}} t_n^{-1/2} \le T^{1/2} t_n^{-\frac{\beta-\nu}{2}}$ (since $k \le t_n \le T$).

\noindent\textbf{Estimate of $J_{F2}$:} Applying the substitutions $\tau = t_{n}-\sigma$ and $j = n - i$, adding and subtracting $F(t_{n-1},X(t_{n-1}))$, and using the triangle inequality yields
\begin{align}
J_{F2}
&= \left\|\sum_{j=1}^{n}\int_{t_{j-1}}^{t_{j}}
\bigl(S(\tau)-S_{h,k}^{j}P_{h}\bigr)F(t_{n-j}, X(t_{n-j}))\,d\tau\right\|_{L^{p}(\Omega;H)} \notag \\
&\leq \left\|\sum_{j=1}^{n}\int_{t_{j-1}}^{t_{j}}
\bigl(S(\tau)-S_{h,k}^{j}P_{h}\bigr)\bigl(F(t_{n-j},X(t_{n-j}))-F(t_{n-1},X(t_{n-1}))\bigr)\,d\tau\right\|_{L^{p}(\Omega;H)} \notag \\
&\quad+ \left\|\sum_{j=1}^{n}\int_{t_{j-1}}^{t_{j}}
\bigl(S(\tau)-S_{h,k}^{j}P_{h}\bigr)F(t_{n-1},X(t_{n-1}))\,d\tau\right\|_{L^{p}(\Omega;H)} 
\eqqcolon J_{F2}^{(1)} + J_{F2}^{(2)}. \label{eq:split_JF2}
\end{align}
To bound $J_{F2}^{(2)}$, we consider two cases based on $\beta$:

\noindent \textbf{Case 1 ($\beta \in (0,1]$):} Applying Lemma~\ref{lem:fully_discrete_error_estimates}(ii) with $\rho = 1$, Assumption~\ref{ass:drift}{\upshape(i)}, and Theorem~\ref{thm:base_regularity}(i) yields, for all $n \ge 1$,
\begin{align}
J_{F2}^{(2)} &\le C\bigl(h + k^{1/2}\bigr)\,\|F(t_{n-1},X(t_{n-1}))\|_{L^{p}(\Omega;\dot{H}^{-1})} \notag \\
&\le C\bigl(h + k^{1/2}\bigr)\bigl(1 + \|X(t_{n-1})\|_{L^{p}(\Omega;H)}\bigr) 
\le C\bigl(h^{\beta} + k^{\frac{\beta}{2}}\bigr) \bigl( 1 + \|X_0\|_{L^p(\Omega; \dot{H}^\nu)} \bigr). \label{eq:bound_JF2_2_case1}
\end{align}

\noindent \textbf{Case 2 ($\beta \in (1,2]$):} We treat $n \ge 2$ and $n = 1$ separately.
\begin{itemize}
    \item \textbf{For $n \ge 2$:} Since $t_{n-1} = \frac{n-1}{n}t_n \ge \frac{1}{2}t_n > 0$, applying Lemma~\ref{lem:fully_discrete_error_estimates}(ii) with $\rho = 2-\beta \in [0,1)$, Assumption~\ref{ass:drift}{\upshape(ii)}, and \eqref{eq:X_H_beta_minus_1_relaxed} (together with \eqref{eq:X_H1_bound_beta2} for $\beta = 2$) gives
    \begin{align}
        J_{F2}^{(2)} &\le C\bigl(h^\beta + k^{\frac{\beta}{2}}\bigr)\,\|F(t_{n-1},X(t_{n-1}))\|_{L^{p}(\Omega;\dot{H}^{\beta-2})} \notag \\
        &\le C\bigl(h^\beta + k^{\frac{\beta}{2}}\bigr)\bigl(1 + \|X(t_{n-1})\|_{L^{p}(\Omega;\dot{H}^{\beta-1})}\bigr) 
        \le C\bigl(h^\beta + k^{\frac{\beta}{2}}\bigr)\, t_{n}^{-\frac{\beta-\nu}{2}} \bigl(1 + \|X_0\|_{L^p(\Omega;\dot{H}^{\nu})}\bigr). \label{eq:bound_JF2_2_case2}
    \end{align}

    \item \textbf{For $n = 1$ ($t_0 = 0, t_1 = k$):} If $\nu > 1$, then $\dot{H}^\nu \hookrightarrow \dot{H}^{\beta-1}$ (since $\beta - 1 \le 1 < \nu$). Consequently, applying Lemma~\ref{lem:fully_discrete_error_estimates}(ii) ($\rho = 2-\beta$) and Assumption~\ref{ass:drift}{\upshape(ii)} at $n=1$ yields
    \begin{align*}
        J_{F2}^{(2)} \big|_{n=1} &\le C\bigl(h^\beta + k^{\frac{\beta}{2}}\bigr) \|F(0,X_0)\|_{L^p(\Omega;\dot{H}^{\beta-2})} 
        \le C\bigl(h^\beta + k^{\frac{\beta}{2}}\bigr) \bigl(1 + \|X_0\|_{L^p(\Omega;\dot{H}^{\beta-1})}\bigr) \\
        &\le C\bigl(h^\beta + k^{\frac{\beta}{2}}\bigr) \bigl(1 + \|X_0\|_{L^p(\Omega;\dot{H}^\nu)}\bigr) 
        \le C\bigl(h^\beta + k^{\frac{\beta}{2}}\bigr) t_1^{-\frac{\beta-\nu}{2}} \bigl(1 + \|X_0\|_{L^p(\Omega;\dot{H}^\nu)}\bigr).
    \end{align*}
  Conversely, if $\nu \le 1$, applying the triangle inequality on $[0,k]$ together with Lemma~\ref{lem:semigroup_estimates}(i), \eqref{eq:negative_norm_smoothing_fully_discrete} ($\eta=1$), and Assumption~\ref{ass:drift}{\upshape(i)} gives
    \begin{align}
        J_{F2}^{(2)} \big|_{n=1} 
        &\le \int_0^k \Bigl(\bigl\|S(\tau)F(0,X_0)\bigr\|_{L^p(\Omega;H)} + \bigl\|S_{h,k}^1 P_h F(0,X_0)\bigr\|_{L^p(\Omega;H)}\Bigr) \,d\tau \notag \\
        &\le C \left(\int_0^k \tau^{-1/2}\,d\tau + k t_1^{-1/2}\right) \|F(0,X_0)\|_{L^p(\Omega;\dot{H}^{-1})} \notag \\
        &\le C k^{1/2} \bigl(1 + \|X_0\|_{L^p(\Omega;H)}\bigr)
        \le C \bigl(h^\beta + k^{\frac{\beta}{2}}\bigr) t_1^{-\frac{\beta-\nu}{2}} \bigl(1 + \|X_0\|_{L^p(\Omega;\dot{H}^\nu)}\bigr), \label{eq:JF2_n1_split}
    \end{align}
    where the last step uses $k^{1/2} = t_1^{\frac{1-\nu}{2}} k^{\frac{\beta}{2}} t_1^{-\frac{\beta-\nu}{2}} \le T^{\frac{1-\nu}{2}} k^{\frac{\beta}{2}} t_1^{-\frac{\beta-\nu}{2}}$ for $\nu \le 1$.
\end{itemize}
Combining \eqref{eq:bound_JF2_2_case1}--\eqref{eq:JF2_n1_split}, we obtain for all $\beta \in (0,2]$ and $n \ge 1$:
\begin{equation}
    J_{F2}^{(2)} \le C\bigl(h^\beta + k^{\frac{\beta}{2}}\bigr) t_n^{-\frac{\beta-\nu}{2}} \bigl(1 + \|X_0\|_{L^p(\Omega;\dot{H}^\nu)}\bigr). \label{eq:bound_JF2_2}
\end{equation}
Next, to bound $J_{F2}^{(1)}$, note that $J_{F2}^{(1)} \equiv 0$ for $n=1$, and for $n \ge 2$ the $j=1$ term vanishes identically. We therefore split $J_{F2}^{(1)} \le I_{a} + I_{b}$, where $I_{a}$ sums over $2 \le j \le n-1$ and $I_{b}$ is the $j=n$ integral over $[t_{n-1}, t_n]$.

To bound $I_{a}$, let $\alpha \coloneqq \max\{\beta-2, -1\}$ and $r \coloneqq \max\{\beta-1, 0\}$, so that $\frac{\beta-\alpha}{2} = \frac{1+\min\{\beta,1\}}{2}$. Applying Lemma~\ref{lem:operator_splitting} (with $\rho = \beta, \eta = \alpha$), Assumption~\ref{ass:drift}{\upshape(i)}--{\upshape(ii)}, and Theorem~\ref{thm:singular_regularity}(iii), and using $t_{n-1} - t_{n-j} = t_{j-1} \le \tau$ and $t_{n-j} \ge \frac{1}{2}(t_n - \tau)$ for $\tau \in [t_{j-1}, t_j]$ ($2 \le j \le n-1$), we obtain
\begin{align}
    I_{a} &\le \sum_{j=2}^{n-1}\int_{t_{j-1}}^{t_{j}} \bigl\|\bigl(S(\tau)-S_{h,k}^{j}P_{h}\bigr)A^{-\frac{\alpha}{2}}\bigr\|_{\mathcal{L}(H)} \bigl\|F(t_{n-j},X(t_{n-j}))-F(t_{n-1},X(t_{n-1}))\bigr\|_{L^p(\Omega;\dot{H}^\alpha)} \,d\tau \notag \\
    &\le C \bigl( h^\beta + k^{\frac{\beta}{2}} \bigr) \sum_{j=2}^{n-1}\int_{t_{j-1}}^{t_{j}} \tau^{-\frac{1+\min\{\beta,1\}}{2}} \Bigl( (t_{n-1}-t_{n-j})^{\frac{\beta}{2}} + \|X(t_{n-1})-X(t_{n-j})\|_{L^p(\Omega;\dot{H}^r)} \Bigr) \,d\tau \notag \\
    &\le C \bigl( h^\beta + k^{\frac{\beta}{2}} \bigr) \bigl(1 + \|X_0\|_{L^p(\Omega;\dot{H}^\nu)}\bigr) \sum_{j=2}^{n-1}\int_{t_{j-1}}^{t_{j}} \tau^{-\frac{1+\min\{\beta,1\}}{2}} \Bigl( \tau^{\frac{\beta}{2}} + \tau^{\frac{\min\{\beta,1\}}{2}} t_{n-j}^{-\frac{\beta-\nu}{2}} \Bigr) \,d\tau \notag \\
    &\le C \bigl( h^\beta + k^{\frac{\beta}{2}} \bigr) \bigl(1 + \|X_0\|_{L^p(\Omega;\dot{H}^\nu)}\bigr) \int_0^{t_n} \tau^{-1/2} \Bigl( 1 + (t_n - \tau)^{-\frac{\beta-\nu}{2}} \Bigr) \, d\tau \notag \\
    &\le C \bigl( h^\beta + k^{\frac{\beta}{2}} \bigr) \, t_n^{-\frac{\beta-\nu}{2}} \bigl(1 + \|X_0\|_{L^p(\Omega;\dot{H}^\nu)}\bigr). \label{eq:bound_JF2_1_Ia}
\end{align}
To bound $I_{b}$ (where $n \ge 2$, so $t_{n-1} \ge t_n/2$), we consider two regimes for $\nu = \min\{\beta, \mu\}$:
\begin{itemize}
    \item \textbf{For $\nu \le 1$:} Applying the triangle inequality together with Lemma~\ref{lem:semigroup_estimates}(i), \eqref{eq:negative_norm_smoothing_fully_discrete} ($\eta = 1$), Assumption~\ref{ass:drift}{\upshape(i)}, and Theorem~\ref{thm:base_regularity}(i) gives
    \begin{align*}
        I_{b} &\le \int_{t_{n-1}}^{t_n} \Bigl(\bigl\|A^{1/2}S(\tau)\bigr\|_{\mathcal{L}(H)} + \bigl\|S_{h,k}^n P_h A^{1/2}\bigr\|_{\mathcal{L}(H)}\Bigr) \bigl(\|F(0,X_0)\|_{L^p(\Omega;\dot{H}^{-1})} + \|F(t_{n-1},X(t_{n-1}))\|_{L^p(\Omega;\dot{H}^{-1})}\bigr) \,d\tau \\
        &\le C \bigl(1 + \|X_0\|_{L^p(\Omega;H)}\bigr) \int_{t_{n-1}}^{t_n} \bigl(\tau^{-1/2} + t_n^{-1/2}\bigr) \,d\tau 
        \le C \bigl( h^\beta + k^{\frac{\beta}{2}} \bigr) t_n^{-\frac{\beta-\nu}{2}} \bigl(1 + \|X_0\|_{L^p(\Omega;\dot{H}^\nu)}\bigr),
    \end{align*}
    where we used $\int_{t_{n-1}}^{t_n} (\tau^{-1/2} + t_n^{-1/2})\,d\tau \le C k t_{n}^{-1/2} \le C k^{\frac{\beta}{2}} t_n^{\frac{1-\nu}{2}} t_n^{-\frac{\beta-\nu}{2}} \le C k^{\frac{\beta}{2}} t_n^{-\frac{\beta-\nu}{2}}$ since $k \le t_n \le T$ and $\nu \le 1$.

    \item \textbf{For $\nu > 1$ (so $\beta \in (1,2]$):} Applying Lemma~\ref{lem:operator_splitting} ($\rho = \beta, \eta = \beta-2$), Assumption~\ref{ass:drift}{\upshape(ii)}, Theorem~\ref{thm:singular_regularity}(v), and $t_{n-1} \le \tau$ yields
    \begin{align*}
        I_{b} &\le \int_{t_{n-1}}^{t_n} \bigl\|\bigl(S(\tau)-S_{h,k}^{n}P_{h}\bigr)A^{\frac{2-\beta}{2}}\bigr\|_{\mathcal{L}(H)} \|F(t_{n-1},X(t_{n-1})) - F(0,X_0)\|_{L^p(\Omega;\dot{H}^{\beta-2})} \,d\tau \\
        &\le C \bigl( h^\beta + k^{\frac{\beta}{2}} \bigr) \bigl(1 + \|X_0\|_{L^p(\Omega;\dot{H}^\nu)}\bigr) \int_{t_{n-1}}^{t_n} \tau^{-1} \Bigl( \tau^{\frac{\beta}{2}} + \tau^{\frac{\nu-\beta+1}{2}} \Bigr) \,d\tau \\
        &\le C \bigl( h^\beta + k^{\frac{\beta}{2}} \bigr) t_n^{-\frac{\beta-\nu}{2}} \bigl(1 + \|X_0\|_{L^p(\Omega;\dot{H}^\nu)}\bigr),
    \end{align*}
    since $\int_{t_{n-1}}^{t_n} \tau^{\gamma-1}\,d\tau \le k t_{n-1}^{\gamma-1} \le C t_n^\gamma \le C t_n^{-\frac{\beta-\nu}{2}}$ for $\gamma \in \bigl\{\frac{\beta}{2}, \frac{\nu-\beta+1}{2}\bigr\} \subset (0,1]$.
\end{itemize}
Combining \eqref{eq:bound_JF2_2}, \eqref{eq:bound_JF2_1_Ia}, and the bound for $I_b$ yields
\begin{equation}\label{eq:bound_JF2}
    J_{F2} \leq C\bigl(h^{\beta}+k^{\frac{\beta}{2}}\bigr)\,t_n^{-\frac{\beta-\nu}{2}} \bigl(1 + \|X_0\|_{L^{p}(\Omega;\dot H^{\nu})}\bigr).
\end{equation}

\noindent\textbf{Estimate of $J_{F1}$:} To handle the singularity at the initial time, we split the sum at the first time step $t_1$:
\begin{align}
    J_{F1}
    &\leq \bigg\| \int_0^{t_1} S(t_{n}-\sigma)\bigl(F(\sigma,X(\sigma))-F(0,X(0))\bigr)\, d\sigma \bigg\|_{L^{p}(\Omega;H)} \notag \\
    &\quad + \bigg\|\sum_{i=1}^{n-1}\int_{t_{i}}^{t_{i+1}}
    S(t_{n}-\sigma)\bigl(F(\sigma,X(\sigma))-F(t_{i},X(t_{i}))\bigr)\,
    d\sigma\bigg\|_{L^{p}(\Omega;H)}. \label{eq:JF1_split}
\end{align}
For the first term in \eqref{eq:JF1_split}, applying Lemma~\ref{lem:semigroup_estimates}(i), Assumption~\ref{ass:drift}{\upshape(i)}, and Theorem~\ref{thm:singular_regularity}(iv) (noting that $\min\{\nu, 1\} \le \beta$ implies $\sigma^{\beta/2} \le C \sigma^{\min\{\nu, 1\}/2}$ for $\sigma \le k \le T$) yields
\begin{align}
    &\bigg\| \int_0^{t_1} S(t_{n}-\sigma)\bigl(F(\sigma,X(\sigma))-F(0,X(0))\bigr)\, d\sigma \bigg\|_{L^{p}(\Omega;H)} \notag \\
    &\quad \leq C \int_0^{t_1} (t_n-\sigma)^{-1/2} \bigl( \sigma^{\beta/2} + \|X(\sigma)-X(0)\|_{L^p(\Omega;H)} \bigr)\,d\sigma \notag \\
    &\quad \leq C \bigl(1 + \|X_0\|_{L^{p}(\Omega;\dot H^{\nu})}\bigr) \int_0^{t_1} (t_n-\sigma)^{-1/2} \sigma^{\frac{\min\{\nu, 1\}}{2}} \,d\sigma \notag \\
    &\quad \leq C k^{\frac{\beta}{2}} t_n^{-\frac{\beta-\nu}{2}} \bigl(1 + \|X_0\|_{L^{p}(\Omega;\dot H^{\nu})}\bigr), \label{eq:JF1_first_term_eval}
\end{align}
where the final step follows by direct evaluation for $n=1$ ($t_1=k$) and $n \ge 2$ ($t_n - \sigma \ge t_n/2$), using $\sigma \le k \le t_n$ and $1 + \min\{\nu, 1\} - \nu \ge 0$ for $\nu \in (0,2]$.

For the remaining sum in \eqref{eq:JF1_split}, applying the fundamental theorem of calculus to $F$ with respect to its second variable yields
\begin{equation}\label{eq:taylors_formula}
    F(\sigma,X(\sigma))-F(t_{i},X(t_{i}))
    = F(\sigma,X(\sigma)) - F(t_i,X(\sigma))
    + F'(t_{i},X(t_{i}))\bigl(X(\sigma)-X(t_{i})\bigr)
    + R_{F}\bigl(t_{i},X(t_{i}),X(\sigma)\bigr),
\end{equation}
where the integral remainder is given by
\begin{equation}\label{eq:taylor_remainder}
    R_{F}\bigl(t_{i},X(t_{i}),X(\sigma)\bigr)
    \coloneqq
    \int_{0}^{1} \Bigl( F'\bigl(t_{i},\,X(t_{i})+\lambda(X(\sigma)-X(t_{i}))\bigr) - F'\bigl(t_{i},X(t_{i})\bigr) \Bigr) \bigl(X(\sigma)-X(t_{i})\bigr) \, d\lambda.
\end{equation}
Expanding $X(\sigma) - X(t_i)$ in the second term on the right-hand side of \eqref{eq:taylors_formula} via
\begin{equation}\label{eq:mild_representation_t_i}
    X(\sigma) - X(t_i)
    = \bigl(S(\sigma-t_{i})-I\bigr)\,X(t_{i})
    + \int_{t_{i}}^{\sigma}S(\sigma-\tau)\,F\bigl(\tau,X(\tau)\bigr)\,d\tau
    + \int_{t_{i}}^{\sigma}S(\sigma-\tau)\,G(\tau)\,dW(\tau),
\end{equation}
and substituting \eqref{eq:JF1_first_term_eval}--\eqref{eq:mild_representation_t_i} into \eqref{eq:JF1_split} yields
\begin{align}
    J_{F1} &\le C k^{\frac{\beta}{2}} t_n^{-\frac{\beta-\nu}{2}} \bigl(1 + \|X_0\|_{L^{p}(\Omega;\dot H^{\nu})}\bigr) \notag \\
    &\quad + \left\|\sum_{i=1}^{n-1}\int_{t_{i}}^{t_{i+1}}
    S(t_{n}-\sigma)\,\bigl(F(\sigma,X(\sigma)) - F(t_i,X(\sigma))\bigr)\,d\sigma\right\|_{L^{p}(\Omega;H)} \notag \\
    & \quad + \left\|\sum_{i=1}^{n-1}\int_{t_{i}}^{t_{i+1}}
    S(t_{n}-\sigma)\,F'(t_{i},X(t_{i}))\,\bigl(S(\sigma-t_{i})-I\bigr)\,X(t_{i})\,d\sigma\right\|_{L^{p}(\Omega;H)} \notag \\
    & \quad+ \left\|\sum_{i=1}^{n-1}\int_{t_{i}}^{t_{i+1}}
    S(t_{n}-\sigma)\,F'(t_{i},X(t_{i}))
    \int_{t_{i}}^{\sigma}
    S(\sigma-\tau)\,F\bigl(\tau,X(\tau)\bigr)\,d\tau\,d\sigma\right\|_{L^{p}(\Omega;H)} \notag \\
    &\quad + \left\|\sum_{i=1}^{n-1}\int_{t_{i}}^{t_{i+1}}
    S(t_{n}-\sigma)\,F'(t_{i},X(t_{i}))
    \int_{t_{i}}^{\sigma}
    S(\sigma-\tau)\,G(\tau)\,dW(\tau)\,d\sigma\right\|_{L^{p}(\Omega;H)} \notag \\
    &\quad + \left\|\sum_{i=1}^{n-1}\int_{t_{i}}^{t_{i+1}}
    S(t_{n}-\sigma)\,
    R_{F}\bigl(t_{i},X(t_{i}),X(\sigma)\bigr)\,
    d\sigma\right\|_{L^{p}(\Omega;H)} \notag \\
    &\eqqcolon C k^{\frac{\beta}{2}} t_n^{-\frac{\beta-\nu}{2}} \bigl(1 + \|X_0\|_{L^{p}(\Omega;\dot H^{\nu})}\bigr) + J_{F1}^{(1)} + J_{F1}^{(2)} + J_{F1}^{(3)} + J_{F1}^{(4)} + J_{F1}^{(5)}. \label{eq:bound_JF1_split5}
\end{align}
For $J_{F1}^{(1)}$, applying Lemma~\ref{lem:semigroup_estimates}(i) and Assumption~\ref{ass:drift}{\upshape(i)} yields
\begin{align}
    J_{F1}^{(1)} 
    &\le \sum_{i=1}^{n-1}\int_{t_{i}}^{t_{i+1}}
    \|A^{1/2} S(t_{n}-\sigma)\|_{\mathcal{L}(H)}\,\|F(\sigma,X(\sigma)) - F(t_i,X(\sigma))\|_{L^{p}(\Omega;\dot{H}^{-1})}\,d\sigma \notag \\
    &\le C \sum_{i=1}^{n-1} \int_{t_i}^{t_{i+1}} (t_n-\sigma)^{-1/2} (\sigma - t_i)^{\frac{\beta}{2}} \, d\sigma 
    \le C k^{\frac{\beta}{2}} \int_{0}^{t_n} (t_n-\sigma)^{-1/2} \, d\sigma 
    \le C k^{\frac{\beta}{2}} t_n^{-\frac{\beta-\nu}{2}} \bigl(1 + \|X_0\|_{L^p(\Omega;\dot{H}^\nu)}\bigr). \label{eq:bound_JF1_1}
\end{align}
For $J_{F1}^{(2)}$, applying Lemma~\ref{lem:semigroup_estimates}(i)--(ii), Assumption~\ref{ass:drift}{\upshape(iii)}, and Theorem~\ref{thm:singular_regularity}(i), and using $\sigma - t_i \le k$ and $t_i \ge \sigma/2$ for $\sigma \in [t_i, t_{i+1}]$ ($i \ge 1$), we obtain
\begin{align}
    J_{F1}^{(2)}
    &\leq C \sum_{i=1}^{n-1}\int_{t_{i}}^{t_{i+1}}
    (t_n-\sigma)^{-1/2} \|\bigl(S(\sigma-t_{i})-I\bigr)A^{-\frac{\beta}{2}}\|_{\mathcal{L}(H)} \|X(t_{i})\|_{L^{p}(\Omega;\dot{H}^\beta)}\, d\sigma \label{eq:JF1_2_operator_bound} \\
    &\leq C k^{\frac{\beta}{2}} \bigl(1 + \|X_{0}\|_{L^{p}(\Omega;\dot H^{\nu})}\bigr) \int_{0}^{t_{n}} (t_n-\sigma)^{-1/2} \sigma^{-\frac{\beta - \nu}{2}} \, d\sigma 
    \leq C k^{\frac{\beta}{2}} \, t_n^{-\frac{\beta - \nu}{2}} \bigl(1 + \|X_{0}\|_{L^{p}(\Omega;\dot H^{\nu})}\bigr). \label{eq:bound_JF1_2}
\end{align}
Similarly, for $J_{F1}^{(3)}$, with $\alpha = \max\{\beta-2, -1\}$, applying Lemma~\ref{lem:semigroup_estimates}(i), Assumption~\ref{ass:drift}, Theorem~\ref{thm:singular_regularity}(i), and $\tau \ge t_i \ge \sigma/2$ for $\tau \in [t_i, \sigma]$ ($i \ge 1$) yields
\begin{align}
    J_{F1}^{(3)}
    &\le C \sum_{i=1}^{n-1}\int_{t_{i}}^{t_{i+1}} (t_n-\sigma)^{-1/2} \int_{t_{i}}^{\sigma} \|A^{-\frac{\alpha}{2}} S(\sigma-\tau)\|_{\mathcal{L}(H)} \, \|F(\tau,X(\tau))\|_{L^{p}(\Omega;\dot{H}^{\alpha})} \,d\tau\,d\sigma \notag \\
    &\le C \bigl(1 + \|X_0\|_{L^p(\Omega;\dot{H}^\nu)}\bigr) \sum_{i=1}^{n-1}\int_{t_{i}}^{t_{i+1}} (t_n-\sigma)^{-1/2} \sigma^{-\frac{\beta-\nu}{2}} \int_{t_{i}}^{\sigma} (\sigma-\tau)^{\frac{\alpha}{2}} \,d\tau\,d\sigma \notag \\
    &\le C k^{\frac{\beta}{2}} \bigl(1 + \|X_0\|_{L^p(\Omega;\dot{H}^\nu)}\bigr) \int_{0}^{t_n} (t_n-\sigma)^{-1/2} \sigma^{-\frac{\beta-\nu}{2}} \,d\sigma \le C k^{\frac{\beta}{2}} \, t_n^{-\frac{\beta-\nu}{2}} \bigl(1 + \|X_0\|_{L^p(\Omega;\dot{H}^\nu)}\bigr), \label{eq:bound_JF1_3}
\end{align}
where we used $(\sigma - t_i)^{1+\alpha/2} \le k^{\max\{\beta,1\}/2} \le C k^{\beta/2}$.

To bound $J_{F1}^{(4)}$, we define the sequence of random variables
\begin{equation*}
    Z_{i} \coloneqq \int_{t_{i}}^{t_{i+1}}
    S(t_{n}-\sigma)\,F'(t_{i},X(t_{i}))\,
    \int_{t_{i}}^{\sigma}
    S(\sigma-\tau)\,G(\tau)\,dW(\tau)\,
    d\sigma, \quad \text{for } i=1, \dots, n-1.
\end{equation*}
Since $\{Z_i\}_{i=1}^{n-1}$ satisfies \eqref{eq:martingale_difference_sequence}, Lemma~\ref{lem:discrete_BDG_inequality} gives   
\begin{equation}\label{eq:bound_JF1_4_intermideate}
    J_{F1}^{(4)}
    = \bigg\|\sum_{i=1}^{n-1}Z_{i}\bigg\|_{L^{p}(\Omega;H)}
    \le C_p \bigg(\sum_{i=1}^{n-1}\|Z_{i}\|_{L^{p}(\Omega;H)}^{2}\bigg)^{\!1/2}.
\end{equation}
To estimate $\|Z_i\|_{L^{p}(\Omega;H)}^{2}$, we apply Minkowski's integral inequality, the continuous Burkholder--Davis--Gundy inequality (Lemma~\ref{lem:BDG}), Lemma~\ref{lem:semigroup_estimates}(i), and Assumptions~\ref{ass:diffusion} and~\ref{ass:drift}{\upshape(iii)}. For each $i = 1, \dots, n-1$, this gives:
\begin{align}
    \|Z_{i}\|_{L^{p}(\Omega;H)}
    &\le \int_{t_{i}}^{t_{i+1}} \|A^{1/2} S(t_n-\sigma)\|_{\mathcal{L}(H)} \bigg\| \int_{t_{i}}^{\sigma} A^{-1/2} F'(t_{i}, X(t_{i}))\, S(\sigma-\tau)\,G(\tau)\,dW(\tau) \bigg\|_{L^{p}(\Omega;H)} \,d\sigma \notag \\
    &\le C \int_{t_{i}}^{t_{i+1}} (t_n-\sigma)^{-1/2} \left( \int_{t_{i}}^\sigma (\sigma-\tau)^{-(1-\beta)^+} \,d\tau \right)^{\!1/2} d\sigma 
    \le C k^{\frac{\min\{\beta,1\}}{2}} \int_{t_{i}}^{t_{i+1}} (t_n-\sigma)^{-1/2} \,d\sigma. \label{eq:bound_Zi_final}
\end{align}
To sum these squared norms (for $n \ge 2$), we split the analysis into interior intervals ($1 \le i \le n-2$) and the final interval ($i = n-1$):
\begin{itemize}
    \item \textbf{For $1 \le i \le n-2$:} Applying the Cauchy--Schwarz inequality on $[t_i, t_{i+1}]$ yields
    \begin{align*}
        \|Z_i\|_{L^p(\Omega;H)}^2 &\le C k^{\min\{\beta,1\}} \left( \int_{t_i}^{t_{i+1}} 1^2 \, d\sigma \right) \left( \int_{t_i}^{t_{i+1}} (t_n-\sigma)^{-1} \, d\sigma \right) \\
        &= C k^{1 + \min\{\beta,1\}} \int_{t_i}^{t_{i+1}} (t_n-\sigma)^{-1} \, d\sigma.
    \end{align*}
    Summing this over $i = 1, \dots, n-2$ bounds the early terms by the logarithmic factor:
    \begin{equation}\label{eq:bound_Zi_sum_early}
    \begin{aligned}
    \sum_{i=1}^{n-2} \|Z_i\|_{L^{p}(\Omega;H)}^{2} 
    &\le C k^{1 + \min\{\beta,1\}} \int_{t_1}^{t_{n-1}} (t_n-\sigma)^{-1} \,d\sigma \\
    & = C k^{1 + \min\{\beta,1\}} \ln\left(\frac{t_n-t_1}{k}\right) \le C k^{1 + \min\{\beta,1\}} \ln(t_n/k).
    \end{aligned}
    \end{equation}
    
    \item \textbf{For $i = n-1$:} Applying Cauchy--Schwarz would create a non-integrable $(t_n-\sigma)^{-1}$ singularity at $t_n$. Instead, we integrate $(t_n-\sigma)^{-1/2}$ directly on $[t_{n-1}, t_n]$ before squaring:
    \begin{equation}\label{eq:bound_Zi_final_n1}
        \|Z_{n-1}\|_{L^{p}(\Omega;H)} 
        \le C k^{\frac{\min\{\beta,1\}}{2}} \int_{t_{n-1}}^{t_n} (t_n-\sigma)^{-1/2} \,d\sigma 
        = C k^{\frac{\min\{\beta,1\}}{2}} \cdot 2k^{1/2} = 2 C k^{\frac{1 + \min\{\beta,1\}}{2}},
    \end{equation}
    which yields $\|Z_{n-1}\|_{L^p(\Omega;H)}^2 \le C k^{1 + \min\{\beta,1\}}$.
\end{itemize}
Combining \eqref{eq:bound_Zi_sum_early} and \eqref{eq:bound_Zi_final_n1}, the total sum is bounded by $C k^{1 + \min\{\beta,1\}} \bigl(1 + \ln(t_n/k)\bigr)$. Substituting this back into \eqref{eq:bound_JF1_4_intermideate}, we obtain
\begin{equation}
    J_{F1}^{(4)} \le C\,k^{\frac{1 + \min\{\beta,1\}}{2}} \bigl(1 + \ln(t_n/k)\bigr)^{1/2}.
\end{equation}
When $\beta \in (0,2)$, the exponent difference $\epsilon \coloneqq \frac{1 + \min\{\beta,1\} - \beta}{2} > 0$, so $k^\epsilon \bigl(1 + \ln(t_n/k)\bigr)^{1/2} \le k^\epsilon \bigl(1 + \ln(T/k)\bigr)^{1/2} \le C$ uniformly for all $k \in (0, T]$; when $\beta = 2$, we have $\frac{1 + \min\{\beta,1\}}{2} = 1 = \frac{\beta}{2}$ and $\bigl(1 + \ln(t_n/k)\bigr)^{1/2} = \ell_{n}^{(2)}$. Since $\beta \ge \nu$ and $t_n \le T$ imply $1 \le C t_n^{-\frac{\beta-\nu}{2}} (1 + \|X_0\|_{L^p(\Omega;\dot{H}^\nu)})$, we conclude for all $\beta \in (0,2]$ that
\begin{equation}\label{eq:bound_JF1_4}
    J_{F1}^{(4)} \le C k^{\frac{\beta}{2}} \ell_{n}^{(\beta)} t_n^{-\frac{\beta-\nu}{2}} \bigl(1 + \|X_0\|_{L^p(\Omega;\dot{H}^\nu)}\bigr).
\end{equation}
For the remainder term $J_{F1}^{(5)}$, from \eqref{eq:taylor_remainder} and Assumption~\ref{ass:drift}{\upshape(iii)}, we have
\begin{align}
    \bigl\|R_{F}(t_i,X(t_{i}),X(\sigma))\bigr\|_{L^{p}(\Omega;\dot{H}^{-\vartheta})}
    &\leq \int_0^1 C \lambda \, \bigl\| \|X(\sigma)-X(t_{i})\|_H^2 \bigr\|_{L^{p}(\Omega)} d\lambda 
    \leq C\,\|X(\sigma)-X(t_{i})\|_{L^{2p}(\Omega;H)}^{2}. \label{eq:bound_RF}
\end{align}
Applying Lemma~\ref{lem:semigroup_estimates}(i), \eqref{eq:bound_RF}, and Theorem~\ref{thm:singular_regularity}(ii) in $L^{2p}(\Omega; H)$, and using $(\sigma-t_i)^\beta t_i^{-(\beta-\nu)} \le C k^{\frac{\beta}{2}} \sigma^{-\frac{\beta}{2}+\nu}$ (since $\sigma-t_i \le \min\{k,\sigma\}$ and $t_i \ge \sigma/2$ for $i \ge 1$), we obtain
\begin{align}
    J_{F1}^{(5)}
    &\leq C\sum_{i=1}^{n-1} \int_{t_{i}}^{t_{i+1}}
    (t_{n}-\sigma)^{-\frac{\vartheta}{2}}\, \bigl\|R_{F}\bigl(t_i,X(t_{i}),X(\sigma)\bigr)\bigr\|_{L^{p}(\Omega;\dot{H}^{-\vartheta})}\, d\sigma \notag \\
    &\leq C \bigl(1 + \|X_0\|_{L^{2p}(\Omega;\dot H^{\nu})}^2\bigr) \sum_{i=1}^{n-1} \int_{t_{i}}^{t_{i+1}}
    (t_{n}-\sigma)^{-\frac{\vartheta}{2}}\,\left( (\sigma-t_i)^\beta \, t_i^{-(\beta -\nu)} + (\sigma-t_i)^{\min\{\beta,1\}} \right)\,d\sigma \notag \\
    &\leq C \bigl(1 + \|X_0\|_{L^{2p}(\Omega;\dot H^{\nu})}^2\bigr) \left( k^{\frac{\beta}{2}} \int_{0}^{t_{n}}
    (t_{n}-\sigma)^{-\frac{\vartheta}{2}}\, \sigma^{-\frac{\beta}{2} + \nu} \,d\sigma + k^{\min\{\beta,1\}} \int_0^{t_n} (t_n-\sigma)^{-\frac{\vartheta}{2}}\,d\sigma \right) \notag \\
    &\leq C\,k^{\frac{\beta}{2}} \, t_n^{-\frac{\beta-\nu}{2}} \left( t_n^{1- \frac{\vartheta}{2} + \frac{\nu}{2} } + k^{\min\{\beta,1\}-\frac{\beta}{2}} t_n^{1-\frac{\vartheta}{2}+\frac{\beta-\nu}{2}} \right) \bigl(1 + \|X_0\|_{L^{2p}(\Omega;\dot H^{\nu})}^2\bigr) \notag \\
    &\leq C\,k^{\frac{\beta}{2}} \, t_n^{-\frac{\beta-\nu}{2}} \bigl(1 + \|X_0\|_{L^{2p}(\Omega;\dot H^{\nu})}^2\bigr). \label{eq:bound_JF1_5}
\end{align}
Collecting the bounds for $J_{F1}^{(1)}$ through $J_{F1}^{(5)}$ and noting that $\|X_0\|_{L^p(\Omega;\dot{H}^\nu)} \le \|X_0\|_{L^{2p}(\Omega;\dot{H}^\nu)}$ yields
\begin{equation}\label{eq:bound_JF1}
    J_{F1} \leq C\,k^{\frac{\beta}{2}} \ell_{n}^{(\beta)} \, t_n^{-\frac{\beta-\nu}{2}} \bigl(1 + \|X_0\|_{L^{2p}(\Omega;\dot H^{\nu})}^2\bigr).
\end{equation}
Finally, substituting \eqref{eq:bound_JF1}, \eqref{eq:bound_JF2}, and \eqref{eq:bound_JF3} into \eqref{eq:bound_JF}, we obtain
\begin{equation}
    J_{F} 
    \leq C\Bigl(h^{\beta}+k^{\frac{\beta}{2}}\ell_{n}^{(\beta)}\Bigr)\,t_n^{-\frac{\beta-\nu}{2}} \bigl(1 + \|X_0\|_{L^{2p}(\Omega;\dot H^{\nu})}^2\bigr) 
    + C\,k\sum_{i=1}^{n-1} t_{n-i}^{-1/2} \|X(t_{i})-X_h^{i}\|_{L^{p}(\Omega;H)}. \label{eq:bound_JF_total}
\end{equation}

\noindent\textbf{Estimate of $J_{G}$:} We apply Lemma~\ref{lem:BDG}, perform the change of variables $\tau = t_{n} - \sigma$ with $j = n - i$, and add and subtract $S(\tau)G(t_{n-j})$ and $\bigl(S(\tau)-S_{h,k}^{j}P_{h}\bigr)G(t_{n-1})$ to obtain
\begin{align}
J_{G}
&\le C\Biggl(\sum_{j=1}^{n}\int_{t_{j-1}}^{t_{j}}
\bigl\|S(\tau)\,\bigl(G(t_{n}-\tau)-G(t_{n-j})\bigr)\bigr\|_{\mathrm{HS}(U_0,H)}^{2} \,d\tau\Biggr)^{\!1/2} \notag \\
&\quad+ C\Biggl(\sum_{j=1}^{n}\int_{t_{j-1}}^{t_{j}}
\bigl\|\bigl(S(\tau)-S_{h,k}^{\,j}P_{h}\bigr) \bigl(G(t_{n-j})-G(t_{n-1})\bigr)\bigr\|_{\mathrm{HS}(U_0,H)}^{2} \,d\tau\Biggr)^{\!1/2} \notag \\
&\quad+ C\Biggl(\sum_{j=1}^{n}\int_{t_{j-1}}^{t_{j}}
\bigl\|\bigl(S(\tau)-S_{h,k}^{\,j}P_{h}\bigr)G(t_{n-1})\bigr\|_{\mathrm{HS}(U_0,H)}^{2} \,d\tau\Biggr)^{\!1/2} \eqqcolon J_{G1} + J_{G2} + J_{G3}. \label{eq:split_JG}
\end{align}
For $J_{G1}$, since $t_n - \tau - t_{n-j} = t_j - \tau \le k$ on $[t_{j-1}, t_j]$ and $\delta \ge \beta/2$, applying Assumption~\ref{ass:diffusion} and Lemma~\ref{lem:semigroup_estimates}(i) yields
\begin{equation}\label{eq:bound_JG1}
J_{G1} \le C \Biggl(\sum_{j=1}^{n}\int_{t_{j-1}}^{t_{j}} \tau^{-(1-\beta)^+} \,(t_{j}-\tau)^{2\delta} \,d\tau\Biggr)^{\!1/2}
\le C\,k^{\frac{\beta}{2}} \Biggl( \int_{0}^{t_n} \tau^{-(1-\beta)^+} \,d\tau \Biggr)^{\!1/2} \le C\,k^{\frac{\beta}{2}}.
\end{equation}
For $J_{G2}$, the $j=1$ term vanishes identically, and for $j \ge 2$ and $\tau \in [t_{j-1}, t_j]$, we have $t_{n-1} - t_{n-j} = t_{j-1} \le \tau$. Applying Lemma~\ref{lem:operator_splitting} (with $\rho=\beta$ and $\eta=\beta-1$) and Assumption~\ref{ass:diffusion} with $\delta \ge \beta/2 > 0$ gives
\begin{equation}\label{eq:bound_JG2}
J_{G2} \leq C\,\bigl(h^{\beta}+k^{\frac{\beta}{2}}\bigr) \Biggl(\sum_{j=2}^{n}\int_{t_{j-1}}^{t_{j}} \tau^{-1}\,(t_{n-1} - t_{n-j})^{2\delta} \,d\tau\Biggr)^{\!1/2}
\leq C\,\bigl(h^{\beta}+k^{\frac{\beta}{2}}\bigr) \Biggl( \int_{0}^{t_n} \tau^{-1+2\delta} \,d\tau \Biggr)^{\!1/2} \leq C \,\bigl(h^{\beta}+k^{\frac{\beta}{2}}\bigr).
\end{equation}
For $J_{G3}$, expanding the Hilbert--Schmidt norm via an orthonormal basis $\{\psi_m\}_{m=1}^\infty$ of $U_0$ and applying Lemma~\ref{lem:fully_discrete_error_estimates}(iii) (with $\rho=\beta$) alongside Assumption~\ref{ass:diffusion} yields
\begin{align}\label{eq:bound_JG3}
J_{G3} &= C \Biggl(\sum_{m=1}^{\infty}\sum_{j=1}^{n}\int_{t_{j-1}}^{t_{j}}
\bigl\|\bigl(S(\tau)-S_{h,k}^{\,j}P_{h}\bigr)\,G(t_{n-1})\psi_m\bigr\|_H^2 \,d\tau\Biggr)^{\!1/2} \notag \\
&\leq C\,\bigl(h^{\beta}+k^{\frac{\beta}{2}}\bigr) \|G(t_{n-1})\|_{\mathrm{HS}(U_0,\dot{H}^{\beta-1})} \leq C\,\bigl(h^{\beta}+k^{\frac{\beta}{2}}\bigr).
\end{align}
Combining \eqref{eq:bound_JG1}--\eqref{eq:bound_JG3} yields $J_G \le C\bigl(h^\beta + k^{\frac{\beta}{2}}\bigr)$. Substituting this together with \eqref{eq:bound_J0} and \eqref{eq:bound_JF_total} into \eqref{eq:bound_fullydiscrete}, and setting $\varphi_m \coloneqq \|X(t_m) - X_h^m\|_{L^p(\Omega;H)}$, we obtain for any fixed $n \in \{1,\dots,N\}$ and all $1 \le m \le n$ (using $\ell_{m}^{(\beta)} \le \ell_{n}^{(\beta)}$),
\begin{equation}\label{eq:bound_fully_discrete_final2}
    \varphi_m \le C\Bigl(h^{\beta}+k^{\frac{\beta}{2}}\ell_{n}^{(\beta)}\Bigr) \Bigl( 1 + t_m^{-\frac{\beta-\nu}{2}} \Bigr) \bigl(1 + \|X_{0}\|_{L^{2p}(\Omega;\dot{H}^{\nu})}^2\bigr) + C\,k\sum_{i=1}^{m-1} t_{m-i}^{-1/2} \varphi_i.
\end{equation}
Applying Lemma~\ref{lem:discrete_gronwall} on $\{1,\dots,n\}$ with $\theta = 1 - \frac{\beta-\nu}{2} > 0$ and $\alpha = \frac{1}{2} > 0$, and using $1 \le T^{\frac{\beta-\nu}{2}} t_n^{-\frac{\beta-\nu}{2}}$ for $t_n \in (0,T]$, we conclude at $m = n$ that
\begin{equation*}
    \|X(t_n) - X_h^n\|_{L^p(\Omega;H)} \le C \Bigl(h^\beta + k^{\frac{\beta}{2}}\ell_{n}^{(\beta)}\Bigr) t_{n}^{-\frac{\beta-\nu}{2}} \Bigl(1 + \|X_{0}\|_{L^{2p}(\Omega;\dot{H}^{\nu})}^2\Bigr),
\end{equation*}
which completes the proof of Theorem~\ref{thm:strong_convergence_of_fully_discrete_scheme}.
\end{proof}

%%%%%%%%%%%%%%%%%%%%%%%%%%%%%%%%%%%%%%%%%%%%%%%%%%%%%%

% \section{Eliminating the logarithmic factor at \texorpdfstring{$\beta = 2$}{beta = 2} in the absence of convection}\label{sec:optimal_error_estimates}

\section{Log-free error estimate at \texorpdfstring{$\beta = 2$}{beta = 2} when \texorpdfstring{$\boldsymbol{b} \equiv \boldsymbol{0}$}{b = 0}}\label{sec:optimal_error_estimates}
%%%%%%%%%%%%%%%%%%%%%%%%%%%%%%%%%%%%%%%%%%%%%%%%%%%%%%

In Theorem~\ref{thm:strong_convergence_of_fully_discrete_scheme}, we established the strong convergence rate $O\bigl(h^\beta + k^{\frac{\beta}{2}}\ell_{n}^{(\beta)}\bigr)$ for $\beta \in (0, 2]$ under Assumption~\ref{ass:drift}, where $F$ and $F'$ map into $\dot{H}^{-1}$ to accommodate the convection term. While this yields the optimal rate $O(h^\beta + k^{\frac{\beta}{2}})$ for all $\beta \in (0,2)$, at the critical endpoint $\beta = 2$ the $\dot{H}^{-1}$-codomain of $F'$ in \eqref{eq:F_prime_bound_H1} introduces the logarithmic factor $\ell_{n}^{(2)} = (1 + \ln(t_n/k))^{1/2}$ in $J_{F1}^{(4)}$. However, when the convection term vanishes ($\boldsymbol{b} \equiv \boldsymbol{0}$ in \eqref{eq:SPDE}), assuming that $F$ and $F'$ map into $H$ rather than $\dot{H}^{-1}$ eliminates this logarithmic factor. Under the same first-order G\^ateaux differentiability framework, we recover the optimal convergence rate $O(h^2 + k)$ away from $t = 0$ without requiring $F$ to be twice differentiable, in contrast to the framework used by Wang~\cite{MR3649432}.

\begin{assumption}\label{ass:drift_stronger} 
Suppose $\boldsymbol{b} \equiv \boldsymbol{0}$ in \eqref{eq:SPDE}. Assume $F \colon [0,T] \times H \to H$ is G\^ateaux differentiable with respect to its second variable, and there exist constants $C>0$ and $\vartheta \in [0,2)$ such that for all $t, s \in [0,T]$ and $u, v, w \in H$,
\begin{align}
\|F(t,u)-F(s,v)\|_H &\le C\,\bigl( |t-s| + \|u-v\|_H\bigr), &
\|F(t,v)\|_H &\le C\bigl(1 + \|v\|_H\bigr), \label{eq:F_stronger_H1} \\
\|F'(t,u)w\|_H &\leq C\,\|w\|_H, &
\|\bigl(F'(t,u)-F'(t,v)\bigr)w\|_{\dot H^{-\vartheta}} &\leq C \,\|u-v\|_H \|w\|_H. \label{eq:F_stronger_H2}
\end{align}
\end{assumption}

\begin{theorem}[{Log-free strong error estimate at $\beta = 2$ when $\boldsymbol{b} \equiv \boldsymbol{0}$}]\label{thm:optimal_strong_convergence}
Suppose Assumptions~\ref{ass:initial_data}--\ref{ass:diffusion} and \ref{ass:drift_stronger} hold with $\beta = 2$ and $\mu \in (0,2]$. Then, for any $p \in [2,\infty)$, there exists a constant $C > 0$, independent of $h$, $k$, $N$, and $n \in \{1, \dots, N\}$, such that the fully discrete approximation \eqref{eq:mild_fully_discrete} satisfies, for all $n=1,\dots,N$,
\begin{equation}\label{eq:optimal-main-convergence}
    \|X(t_n) - X_h^n\|_{L^p(\Omega;H)} \le C \bigl(h^2 + k\bigr) t_{n}^{-\frac{2-\mu}{2}} \left(1 + \|X_{0}\|_{L^{2p}(\Omega;\dot H^{\mu})}^2\right).
\end{equation}
\end{theorem}

\begin{proof}
Since Assumption~\ref{ass:drift_stronger} implies Assumption~\ref{ass:drift} at $\beta = 2$ (with $\nu = \min\{2,\mu\} = \mu$), Theorem~\ref{thm:singular_regularity} and the bounds for $J_0$, $J_{F3}$, $J_{F2}$, $J_{F1}^{(1)}$--$J_{F1}^{(3)}$, $J_{F1}^{(5)}$, and $J_G$ in the proof of Theorem~\ref{thm:strong_convergence_of_fully_discrete_scheme} remain valid at $\beta = 2$. It therefore suffices to re-estimate $J_{F1}^{(4)}$. Using \eqref{eq:F_stronger_H2} for $F'(t_i, X(t_i))$ together with Lemma~\ref{lem:BDG}, and Assumption~\ref{ass:diffusion} (since $\dot{H}^1 \hookrightarrow H$) avoids extracting $A^{1/2}$ in \eqref{eq:bound_Zi_final}, yielding
\begin{align*}
    \|Z_i\|_{L^p(\Omega;H)} 
    &\le \int_{t_i}^{t_{i+1}} \|S(t_n-\sigma)\|_{\mathcal{L}(H)} \bigg\| \int_{t_i}^\sigma F'(t_i, X(t_i)) S(\sigma-\tau) G(\tau)\,dW(\tau) \bigg\|_{L^p(\Omega;H)} d\sigma \\
    &\le C \int_{t_i}^{t_{i+1}} \left( \int_{t_i}^\sigma \|F^\prime(t_{i},X(t_{i}))\, S(\sigma-\tau)\,G(\tau)\|_{\mathrm{HS}(U_0,H)}^{2} \, d\tau \right)^{\!1/2} d\sigma \\
    &\le C \int_{t_i}^{t_{i+1}} (\sigma - t_i)^{1/2}\,d\sigma \le C k^{3/2}.
\end{align*}
Substituting this bound into \eqref{eq:bound_JF1_4_intermideate} gives
\[
    J_{F1}^{(4)} \le C_p \left( \sum_{i=1}^{n-1} \|Z_i\|_{L^p(\Omega;H)}^2 \right)^{\!1/2} \le C \bigl((n-1) k^3\bigr)^{1/2} \le C k \, t_n^{-\frac{2-\mu}{2}} \bigl(1 + \|X_0\|_{L^p(\Omega;\dot{H}^\mu)}\bigr),
\]
which removes the logarithmic factor $\ell_{n}^{(2)}$. Applying Lemma~\ref{lem:discrete_gronwall} as in \eqref{eq:bound_fully_discrete_final2} completes the proof.
\end{proof}

%%%%%%%%%%%%%%%%%%%%%%%%%%%%%%%%%%%%%%%
\section{Numerical experiments}\label{sec:numerical_results}
%%%%%%%%%%%%%%%%%%%%%%%%%%%%%%%%%%%%%%%

In this section, we present numerical experiments to verify the theoretical strong convergence rates established in Sections~\ref{sec:fullydiscrete_approximation} and~\ref{sec:optimal_error_estimates}.

\paragraph{Problem setting and noise regularity.} We set $\mathcal{O} \coloneqq (0,1)$, $T = 1$, and $U = H = L^2(\mathcal{O})$. We take $A = -\frac{\partial^2}{\partial x^2}$ on $\mathcal{D}(A) = H^2(\mathcal{O}) \cap H_0^1(\mathcal{O})$, which admits eigenpairs $\eta_j = \pi^2 j^2$ and $e_j(x) = \sqrt{2}\sin(j\pi x)$ for $j \in \mathbb{N}$. The driving $Q$-Wiener process $W$ is represented via the Karhunen--Lo\`eve expansion
\begin{equation}\label{eq:KL_expansion}
    W(t,x) = \sum_{j=1}^{\infty} \sqrt{q_j}\, e_j(x) \beta_j(t),
\end{equation}
where $\{\beta_j(t)\}_{j\in\mathbb{N}}$ are independent standard Brownian motions. To control the spatial regularity of the noise, we define the covariance operator $Q = A^{-r}$ for $r \ge 0$, which yields eigenvalues $q_j = (\pi j)^{-2r}$. Setting $G(t) \equiv \mathrm{Id}_H$, Assumption~\ref{ass:diffusion} requires the following Hilbert--Schmidt norm to be finite:
\begin{equation}\label{eq:hs_derivation}
    \|\mathrm{Id}_H\|_{\mathrm{HS}(U_0, \dot{H}^{\beta-1})}^2 = \sum_{j=1}^{\infty} \bigl\| A^{\frac{\beta-1}{2}} Q^{1/2} e_j \bigr\|_H^2 = \sum_{j=1}^{\infty} (\pi j)^{2\beta-2-2r} = \pi^{2\beta-2-2r} \sum_{j=1}^{\infty} j^{2\beta-2-2r}.
\end{equation}
This series converges if and only if $2r > 2\beta - 1$. Setting $2r = 2\beta - 1 + \epsilon$ with $\epsilon = 0.001$ reduces \eqref{eq:hs_derivation} to the convergent series $\pi^{-1-\epsilon} \sum_{j=1}^{\infty} j^{-1-\epsilon} < \infty$, ensuring that Assumption~\ref{ass:diffusion} holds. In our simulations, we consider five noise regularity regimes $\beta \in \{2, 1.99, 1.5, 1, 0.4995\}$, which correspond respectively to $r \in \{1.5005, 1.4905, 1.0005, 0.5005, 0\}$ (where $r = 0$ represents space-time white noise, $Q = \mathrm{Id}_H$).

\paragraph{Discretization and error metrics.}
We discretize \eqref{eq:SEE} using continuous piecewise linear finite elements on a uniform spatial mesh of $N_x$ subintervals ($h = 1/N_x$) and the linearly implicit Euler scheme with $N_t$ uniform time steps ($k = T/N_t$).

Expectations are approximated via Monte Carlo averaging over $N_{\text{sim}} = 250$ independent trajectories against a reference solution $X_{\text{ref}}$ computed on a fine grid with $N_{x,\text{ref}} = 256$ and $N_{t,\text{ref}} = 65536$. In the simulations, the Karhunen--Lo\`eve expansion \eqref{eq:KL_expansion} is truncated at $N_{x,\text{ref}} - 1$ modes, matching the dimension of the reference finite element space (cf.\ Yan~\cite[Section~4]{MR2182132}). We evaluate five nested coarse refinement levels $N_x = 2^{l+1}$ ($l = 1,\dots,5$) coupled via $N_t = N_x^2$ ($k = h^2$). We measure the strong error in the $L^\infty(0,T; L^2(\Omega; H))$ norm. To quantify the approximation accuracy at discrete times $t_n = nk$, we compute both the standard maximum strong error $e_{\max}(h,k)$ and the time-weighted strong error $e_{w,\max}(h,k)$ (which incorporates a linear weight $t_n$ to absorb the initial singularity $t_n^{-(\beta-\nu)/2}$ near $t = 0$, since $\beta - \nu < 2$):
\begin{equation}\label{eq:error_metrics}
    e_{\max}(h,k) \approx \max_{1 \le n \le N_t} \mathcal{E}_n, \quad
    e_{w,\max}(h,k) \approx \max_{1 \le n \le N_t} \bigl(t_n \mathcal{E}_n\bigr), \quad
    \mathcal{E}_n \coloneqq \left( \frac{1}{N_{\text{sim}}} \sum_{m=1}^{N_{\text{sim}}} \bigl\| X_h^n(\omega_m) - X_{\text{ref}}(t_n, \omega_m) \bigr\|_H^2 \right)^{\!1/2}.
\end{equation}
Between successive refinement levels $l$ and $l+1$ ($l = 1,\dots,4$), the empirical spatial and temporal convergence orders are defined by
\[
    R_x = \frac{\log\bigl(e_{\max}(h_l, k_l) / e_{\max}(h_{l+1}, k_{l+1})\bigr)}{\log(h_l / h_{l+1})}, \qquad
    R_t = \frac{\log\bigl(e_{\max}(h_l, k_l) / e_{\max}(h_{l+1}, k_{l+1})\bigr)}{\log(k_l / k_{l+1})},
\]
with the time-weighted rates $\tilde{R}_x$ and $\tilde{R}_t$ defined analogously by replacing $e_{\max}$ with $e_{w,\max}$. Note that the parabolic coupling $k_l = h_l^2$ implies $R_x = 2 R_t$ and $\tilde{R}_x = 2 \tilde{R}_t$. All numerical simulations were implemented in MATLAB.

%%%%%%%%%%%%%%%%%%%%%%%%%%%%%%%%%%
\subsection{Numerical test cases and discussion}\label{subsec:numerical_examples}
%%%%%%%%%%%%%%%%%%%%%%%%%%%%%%%%%%

To examine the interplay between the noise regularity $\beta$ and the initial data regularity $\mu$, we consider the one-dimensional semilinear SPDE driven by additive noise on $\mathcal{O} = (0,1)$ for $t \in (0,1]$:
\begin{equation}\label{eq:1d_additive_test_pde}
\left\{
\begin{aligned}
    du(t,x) &= \left[ \frac{\partial^2 u}{\partial x^2}(t,x) + f\bigl(u(t,x)\bigr) \right] dt + dW(t,x), && (t,x) \in (0,1] \times (0,1), \\
    u(t,0) &= u(t,1) = 0, && t \in (0,1], \\
    u(0,x) &= u_0(x), && x \in (0,1).
\end{aligned}
\right.
\end{equation}
Below, we examine five test cases by varying the nonlinear drift $f$ and the initial condition $u_0$.
\begin{example}[Smooth initial data with modified Langmuir drift]\label{ex:modified_langmuir_smooth}
We first set $f(u) = \frac{u}{1 + |u|}$ and choose the smooth initial profile $u_0(x) = \sin(\pi x) \in \dot{H}^2(\mathcal{O})$ ($\mu = 2$). Here $\nu = \min\{\beta, \mu\} = \beta$, so the singular weight $t_n^{-(\beta-\nu)/2} \equiv 1$ is absent and the convergence rate is governed solely by $\beta$. Under the coupled refinement $N_t = N_x^2$, Figure~\ref{fig:additive_ex1} confirms the theoretical bound $O\big(h^\beta + k^{\frac{\beta}{2}}\big)$ from Theorem~\ref{thm:strong_convergence_of_fully_discrete_scheme}, displaying spatial slopes of $\beta$ (left column) and temporal slopes of $\beta/2$ (right column) in both the unweighted metric $e_{\max}$ and the time-weighted metric $e_{w,\max}$.
\end{example}

\begin{figure}[htbp]
    \centering
    \includegraphics[width=0.88\textwidth]{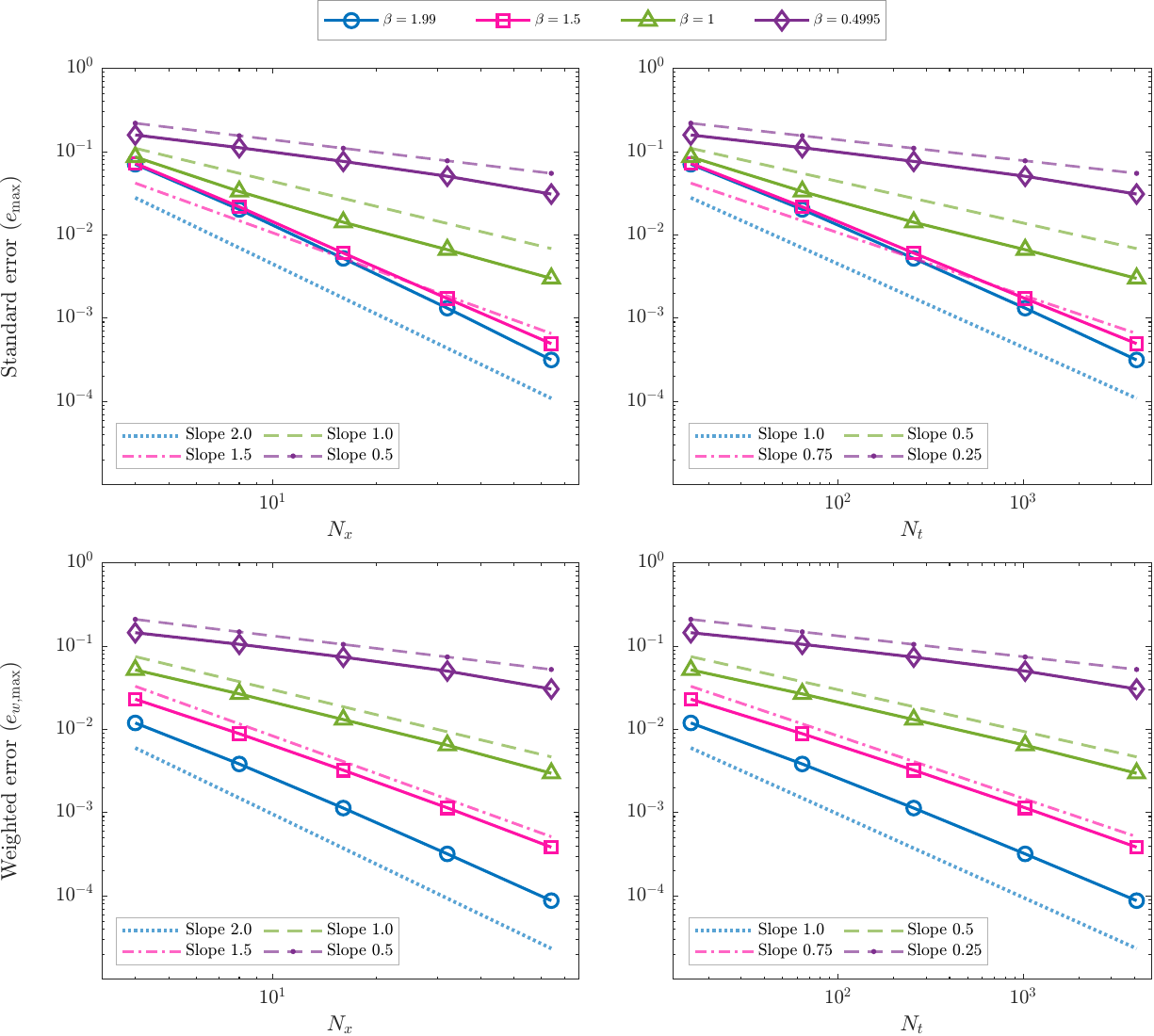}
    \caption{Empirical spatial (left) and temporal (right) convergence orders for Example~\ref{ex:modified_langmuir_smooth} ($u_0(x) = \sin(\pi x)$) under the unweighted metric $e_{\max}$ (top) and time-weighted metric $e_{w,\max}$ (bottom).}
    \label{fig:additive_ex1}
\end{figure}

\begin{example}[Discontinuous initial data with modified Langmuir drift]\label{ex:langmuir_indicator}
We retain the modified Langmuir drift $f(u) = \frac{u}{1 + |u|}$ and consider the discontinuous initial condition $u_0(x) = \mathbbm{1}_{(0, 0.5)}(x) \in \dot{H}^{1/2-\epsilon}(\mathcal{O})$ for $\epsilon > 0$ ($\mu = 1/2 - \epsilon$). As seen in the top row of Figure~\ref{fig:additive_ex2}, the lack of initial regularity causes order reduction in the unweighted error $e_{\max}$, limiting the convergence rate to approximately $O(h^{0.5} + k^{0.25})$ for all $\beta \ge 1$. However, multiplying the error by the time weight $t_n$ in $e_{w,\max}$ (bottom row of Figure~\ref{fig:additive_ex2}) compensates for the initial singularity $t_n^{-(\beta-\nu)/2}$ near $t = 0$ and recovers the sharp theoretical rate $O\big(h^\beta + k^{\frac{\beta}{2}}\big)$ across all four values of $\beta$.
\end{example}

\begin{figure}[htbp]
    \centering
    \includegraphics[width=0.88\textwidth]{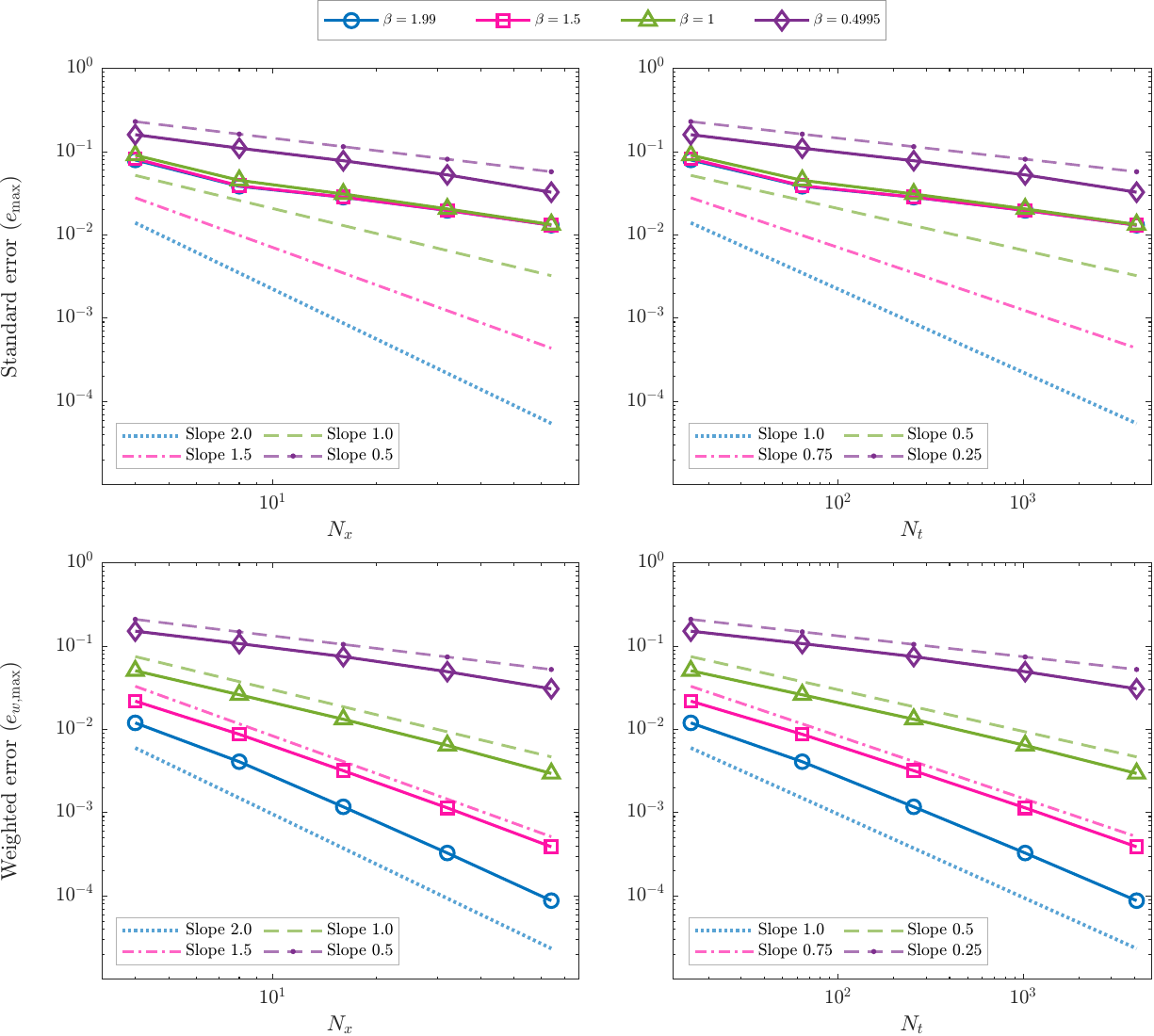}
    \caption{Empirical spatial (left) and temporal (right) convergence orders for Example~\ref{ex:langmuir_indicator} ($u_0(x) = \mathbbm{1}_{(0,0.5)}(x)$) under the unweighted metric $e_{\max}$ (top) and time-weighted metric $e_{w,\max}$ (bottom).}
    \label{fig:additive_ex2}
\end{figure}

\begin{example}\label{ex:nonlinear_fractional_drift}
For our third test case, we set $f(u) = \frac{u|u|}{1 + u^2}$ and prescribe the nonsmooth initial profile $u_0(x) = \sin(\pi \sqrt{x})$ in \eqref{eq:1d_additive_test_pde}, which belongs to $\dot{H}^{1-\epsilon}(\mathcal{O})$ for any $\epsilon > 0$ ($\mu = 1 - \epsilon$). As shown in the top row of Figure~\ref{fig:additive_ex3}, this initial regularity barrier restricts the unweighted error $e_{\max}$ to $O(h^{1-\epsilon} + k^{\frac{1-\epsilon}{2}})$ whenever $\beta > 1$ ($\beta \in \{1.5, 1.99\}$). By contrast, for $\beta \le 1$ ($\beta \in \{1, 0.4995\}$), the roughness of the noise dominates the error and the scheme achieves its natural convergence rate without order reduction. When evaluated under the time-weighted metric $e_{w,\max}$ (bottom row of Figure~\ref{fig:additive_ex3}), the weight $t_n$ suppresses the initial singularity at $t = 0$ and recovers the sharp $O\big(h^\beta + k^{\frac{\beta}{2}}\big)$ convergence rate across all four regimes of $\beta$.
\end{example}

\begin{figure}[htbp]
    \centering
    \includegraphics[width=0.88\textwidth]{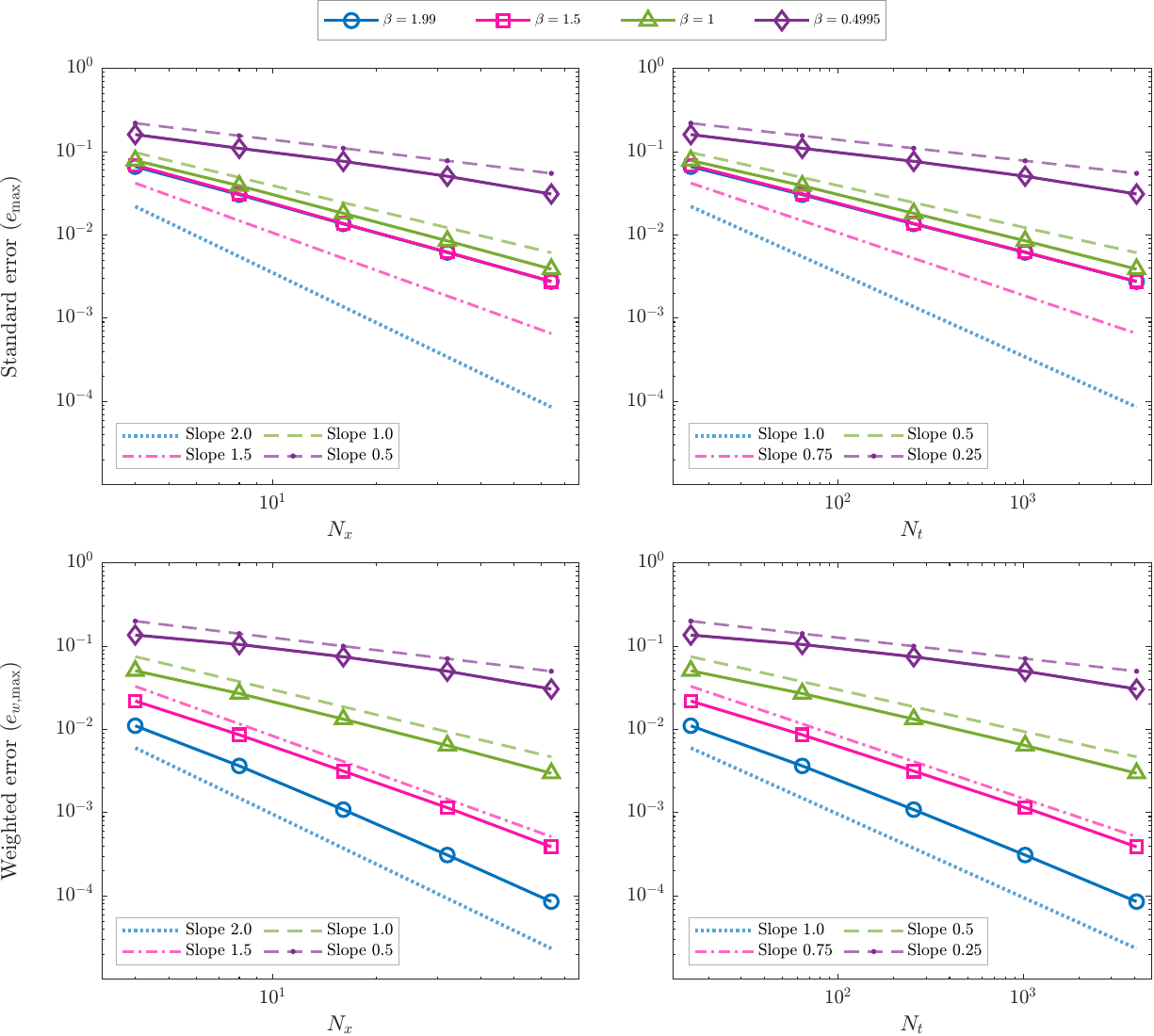}
    \caption{Empirical spatial (left) and temporal (right) convergence orders for Example~\ref{ex:nonlinear_fractional_drift} ($u_0(x) = \sin(\pi\sqrt{x})$) under the unweighted metric $e_{\max}$ (top) and time-weighted metric $e_{w,\max}$ (bottom).}
    \label{fig:additive_ex3}
\end{figure}

\begin{example}\label{ex:sine_gordon_step}
Next, to verify the optimal convergence rate of $O(h^2 + k)$ at the critical endpoint $\beta = 2$ established in Theorem~\ref{thm:optimal_strong_convergence}, we set $f(u) = \sin(u)$ and $u_0(x) = \mathbbm{1}_{(0, 0.5)}(x)$ in \eqref{eq:1d_additive_test_pde}. The smooth sinusoidal nonlinearity $f(u) = \sin(u)$ satisfies Assumption~\ref{ass:drift_stronger} for $\beta = 2$ and the discontinuous initial datum satisfies $u_0 \in \dot{H}^{1/2-\epsilon}(\mathcal{O})$ for $\epsilon > 0$ ($\mu = 1/2 - \epsilon$). As reported in Table~\ref{tab:nonsmooth_step_convergence_beta2}, this nonsmooth initial state reduces the unweighted convergence orders to $R_x \approx 0.5$ and $R_t \approx 0.25$. Introducing the linear time weight $t_n$ in $e_{w,\max}$ neutralizes the initial singularity $t_n^{-(2-\mu)/2}$ and restores the optimal $O(h^2 + k)$ rate ($\tilde{R}_x \approx 1.94$ and $\tilde{R}_t \approx 0.97$) predicted by Theorem~\ref{thm:optimal_strong_convergence}.
\end{example}

\begin{table}[htbp]
\centering
\caption{Strong approximation errors and empirical convergence rates for Example~\ref{ex:sine_gordon_step} at the critical endpoint $\beta=2$ ($N_{\text{sim}}=250$, $N_{x,\text{ref}}=256$, $N_{t,\text{ref}}=65536$).}
\label{tab:nonsmooth_step_convergence_beta2}
\begin{tabular}{@{} l l c c c c c c @{}}
\toprule
\multirow{2}{*}{$N_t$} & \multirow{2}{*}{$N_x$} & \multicolumn{3}{c}{Standard Error Metric} & \multicolumn{3}{c}{Time-Weighted Metric ($t_n^1$)} \\
\cmidrule(lr){3-5} \cmidrule(l){6-8}
& & $e_{\max}(h,k)$ & $R_t$ & $R_x$ & $e_{w,\max}(h,k)$ & $\tilde{R}_t$ & $\tilde{R}_x$ \\
\midrule
16   & 4  & 8.135e-02 & --    & --    & 1.173e-02 & --    & --    \\ 
64   & 8  & 3.888e-02 & 0.533 & 1.065 & 3.855e-03 & 0.803 & 1.605 \\ 
256  & 16 & 2.852e-02 & 0.223 & 0.447 & 1.149e-03 & 0.873 & 1.747 \\ 
1024 & 32 & 1.976e-02 & 0.265 & 0.530 & 3.253e-04 & 0.910 & 1.820 \\ 
4096 & 64 & 1.313e-02 & 0.295 & 0.590 & 8.506e-05 & 0.968 & 1.935 \\ 
\bottomrule
\end{tabular}
\end{table}

\begin{example}\label{ex:c11_piecewise_step}
Finally, to demonstrate that the optimal convergence rate of $O(h^2 + k)$ at the critical endpoint $\beta = 2$ (Theorem~\ref{thm:optimal_strong_convergence}) holds for once-differentiable nonlinearities whose second derivative fails to exist, we consider \eqref{eq:1d_additive_test_pde} with the discontinuous initial condition $u_0(x) = \mathbbm{1}_{(0, 0.5)}(x)$ and the nonlinearity
\begin{equation}\label{eq:piecewise_c11_nonlinearity}
    f(u) = \begin{cases}
        u - \dfrac{u|u|}{2\varepsilon}, & |u| \le \varepsilon, \\[6pt]
        \dfrac{\varepsilon}{2}\,\operatorname{sign}(u), & |u| > \varepsilon,
    \end{cases}
\end{equation}
with $\varepsilon = 0.1$, which satisfies Assumption~\ref{ass:drift_stronger} since $f'(u) = \max\{1 - |u|/\varepsilon, 0\}$ is bounded and Lipschitz continuous. As reported in Table~\ref{tab:c11_piecewise_step_beta2}, the jump discontinuity in $u_0 \in \dot{H}^{1/2-\epsilon}(\mathcal{O})$ ($\epsilon > 0$) reduces the unweighted convergence orders to $R_x \approx 0.59$ and $R_t \approx 0.30$. Introducing the linear time weight $t_n$ in $e_{w,\max}$ neutralizes the initial singularity $t_n^{-(2-\mu)/2}$ and restores the optimal $O(h^2 + k)$ convergence rate ($\tilde{R}_x \approx 1.94$ and $\tilde{R}_t \approx 0.97$) predicted by Theorem~\ref{thm:optimal_strong_convergence}.
\end{example}

\begin{table}[htbp]
\centering
\caption{Strong approximation errors and empirical convergence rates for Example~\ref{ex:c11_piecewise_step} at the critical endpoint $\beta=2$ ($N_{\text{sim}}=250$, $N_{x,\text{ref}}=256$, $N_{t,\text{ref}}=65536$).}
\label{tab:c11_piecewise_step_beta2}
\begin{tabular}{@{} l l c c c c c c @{}}
\toprule
\multirow{2}{*}{$N_t$} & \multirow{2}{*}{$N_x$} & \multicolumn{3}{c}{Standard Error Metric} & \multicolumn{3}{c}{Time-Weighted Metric ($t_n^1$)} \\
\cmidrule(lr){3-5} \cmidrule(l){6-8}
& & $e_{\max}(h,k)$ & $R_t$ & $R_x$ & $e_{w,\max}(h,k)$ & $\tilde{R}_t$ & $\tilde{R}_x$ \\
\midrule
16   & 4  & 7.829e-02 & --    & --    & 1.110e-02 & --    & --    \\ 
64   & 8  & 3.833e-02 & 0.515 & 1.030 & 3.674e-03 & 0.798 & 1.596 \\ 
256  & 16 & 2.846e-02 & 0.215 & 0.429 & 1.097e-03 & 0.872 & 1.743 \\ 
1024 & 32 & 1.975e-02 & 0.264 & 0.527 & 3.138e-04 & 0.903 & 1.806 \\ 
4096 & 64 & 1.312e-02 & 0.295 & 0.590 & 8.194e-05 & 0.969 & 1.937 \\ 
\bottomrule
\end{tabular}
\end{table}

% ==========================================
% ACKNOWLEDGMENTS SECTION
% ==========================================

\paragraph*{Acknowledgments.} The authors acknowledge the support provided by the Indian Institute of Technology Goa, India.

\paragraph*{Funding.} No specific funding was received for this work.

% ==========================================
% REFERENCES SECTION
% ==========================================
% Add the References section to the Table of Contents
\addcontentsline{toc}{section}{References}

%%%%%%% 1 To get just number in the citation
% \bibliographystyle{plain}
\bibliographystyle{abbrv} %% Shortens first names, journal names, and months.
%%%%%%% 2 To get name and year in the citation
% \bibliographystyle{plainnat} 
\bibliography{references.bib}

\appendix

\end{document}